\documentclass{article}
\usepackage{amsmath, longtable, amssymb, amsthm, tikz}
\usepackage[all]{xy}
\usepackage[colorlinks = true]{hyperref}

\newtheorem{example}{Example}
\newtheorem{definition}{Definition}
\newtheorem{theorem}{Theorem}
\newtheorem{lemma}{Lemma}
\newtheorem{corollary}{Corollary}
\newtheorem{procedure}{Procedure}

\title{Constructing Non-isomorphic Signless Laplacian Cospectral Graphs}
\author{Supriyo Dutta \\ Department of Mathematics \\ Indian Institute of Technology Jodhpur \\ Email: \texttt{dosupriyo@gmail.com}}
\date{}

\begin{document}

	\maketitle
		
	\begin{abstract}
		In this article, we generate large families of non-isomorphic and signless Lalacian cospectral graphs using partial transpose on graphs. Our constructions are significantly powerful. More than $70\%$ of non-isomorphic signless-Laplacian cospectral graphs can be generated with partial transpose when number of vertices is $\le 8$. We have also produced numerous examples of non-isomorphic signless Laplacian cospectral graphs.
	\end{abstract}

	\section{Introduction}
		
		The graph isomorphism problem is a long standing open problem in graph theory. Two graphs $G = (V(G), E(G))$ and $H = (V(H), E(H))$ are isomorphic if there is a bijective mapping $f: V(G) \rightarrow V(H)$ such that $(f(u), f(v)) \in E(H)$ if and only if $(u, v) \in E(G)$. The graph isomorphism problem is to determine whether two given graphs are isomorphic or not. This problem was initially attempted with the help of graph spectra. The spectral graph theory elaborates the properties of graphs and eigenvalues of a matrix $M$ related to the graph. For instance, $M$ may be the adjacency matrix, Laplacian matrix, and signless Laplacian matrix. The spectra of a matrix $M$ is the multiset of its eigenvalues, which is denoted by $\Lambda(M)$. The $M$-spectra of a graph is the spectra of the corresponding $M$ matrix. Graphs with equal $M$-spectra are called $M$-cospectral. A graph is determined by its $M$-spectra if there is no other non-isomorphic graph with equal $M$-spectra.
		
		A central question in spectral graph theory \cite{bapat2010graphs} arises to know the sets of graphs which are determined by their eigenvalues \cite{van2003graphs}. This question was originated from Chemistry. Initially, it was believed that every graph is determined by its eigenvalues. But, a pair of cospectral trees was exhibited \cite{von1957spektren} in 1956. Nowadays a number of constructions of cospectral graphs are known for different $M$. A detailed list is available in the reference of \cite{van2003graphs}. Computer estimation suggests that almost all the graphs are determined by their eigenvalues. But till date there is no efficient method to construct all non-isomorphic graph of a given order. Hence, there is a scope of research to develop new methods in this field which is expected to be accepted. Another important motivation to this problem comes from complexity theory. It is still unknown whether the graph isomorphism problem is computationally a hard or easy problem, in general. But checking whether two graphs are cospectral can be done in polynomial time. Recent works in this direction \cite{babai2016graph} renew the interest for these questions. 
		
		It was believed that the eigenvalues of signless Laplacian matrix is more efficient in studying properties of graphs than other matrices \cite{cvetkovic2009towards}. The signless Laplacian matrix $Q(G)$ of a graph $G$ is defined by $Q(G) = D(G) + A(G)$, where $A(G)$ and $D(G)$ are the adjacency, and degree matrices, respectively. In case of adjacency matrix, there are a number of well known methods for generating non-isomorphic cospectral graphs in literature, for instance, Godsil McKay switching \cite{godsil1982constructing}. With the help of product graphs, it can be shown that exponentially large classes of non-isomorphic $Q$-cospectral graphs exist \cite{carvalho2017exponentially}. 
		
		In quantum mechanics and information theory we use the idea of Partial Transpose (PT) \cite{peres1996, horodecki1997} for detecting entanglement. A graph theoretic counterpart of partial transpose was developed in \cite{wu2006conditions} and further developed by \cite{hildebrand2008combinatorial, dutta2016bipartite}. It initiate another idea of graph switching which is  foundationally different from Godsil-McKay switching. As far as our knowledge, it is not a variant of any other switching techniques available in the literature. Earlier, we have employed this method for generating cospectral graphs with respect to the adjacency matrices \cite{dutta2018construction}. Here, we find an efficient method for generating large classes of non-isomorphic $Q$-cospectral graphs using partial transpose. It is a promising candidate in this ground as it generates more than $70\%$ of these graphs when $|V(G)| \leq 8$. Also, these graphs follow a particular pattern which can be easily generalised for higher ordered graphs. Here, we utilize the connections between partial transpose and TU subgraphs of a graph. It makes this work purely graph theoretic and a number of constructions have no trivial matrix counterpart. 
		
		This article is distributed as follows. In the section 2, we briefly discuss all preliminary ideas related to this article. Here, we shall mainly concentrate on the coefficients of the characteristic polynomial of signless Laplacian matrix in terms of TU subgraphs. In the section 3, we introduce the idea of partial transpose of a graph and we describe a number of its properties to provide a clear idea of this switching to the readers. We compare partial transpose with Godsil-McKay switching. How many non-isomorphic $Q$-cospectral graphs are determined by partial transpose? We provide an estimate in the section 4. In the section 5, we state a number of theorems for generating these graphs. Every theorem follows a particular pattern in the structures of generated graphs. Then we conclude with a number of future problems in this direction.

	\section{The coefficients of Q-polynomial}
		
		Throughout this article $n$ and $m$ denote the number of vertices and the number of edges of a graph, respectively. Eigenvalues of a matrix are roots of its characteristic equation. Here, we call the characteristic polynomial of $Q(G)$ as the $Q$-polynomial of the graph $G$ which is denoted and defined by,
		\begin{equation}
		Q_G(\lambda) = \operatorname{det}(Q(G) - \lambda I) = \sum_{j = 0}^n p_j \lambda^{n - j} = p_0 \lambda^n + p_1 \lambda^{n - 1} + \dots p_n.
		\end{equation}
		The union of two given graphs $G_1 = (V(G_1), E(G_1))$ and $G_2 = (V(G_2), E(G_2))$ is denoted by $G_1 \cup G_2$ consists of a vertex set $V(G_1 \cup G_2) = V(G_1) \cup V(G_2)$ and an edge set $E(G_1 \cup G_2) = E(G_1) \cup E(G_2)$. If $G$ can be expressed as $G = G_1 \cup G_2$ then $Q_G(\lambda) = Q_{G_1}(\lambda) Q_{G_2}(\lambda)$. 
		
		A cycle in a graph is a finite sequence of distinct vertices $\delta = (v_1, v_2, \dots v_{|\delta|})$ such that $(v_i, v_{i + 1}) \in E(G)$ for $i = 1, 2, \dots (|\delta| - 1)$ and $(v_\delta, v_1) \in E(G)$. Here $|\delta|$ denotes the length of cycle $\delta$. A spanning subgraph of $G$ whose components are trees or odd unicyclic graph is called a TU subgraph of $G$. Let there be a TU-subgraph $H$ of $G$ containing $c$ unicyclic graphs, as well as trees $T_1, T_2, \dots T_s$. The weight $W(H)$ of $H$ is defined by \cite{guo2017coefficients, cvetkovic2007signless},
		\begin{equation}\label{TU_weight}
		W(H) = 4^c\prod_{i = 1}^s(1 + |E(T_i)|),
		\end{equation}
		where $|E(T_i)|$ is the number of edges in the tree $T_i$. Let $H_j$ be TU subgraphs containing $j$ edges for $j = 1, 2, \dots m$. It is proved that, if $m \ge n$, then, $p_0 = 1$ and,
		\begin{equation} \label{char_coeff}
		p_j = \sum_{H_j}(-1)^j W(H_j), ~\text{for}~ j = 1, 2, \dots n,
		\end{equation}
		where the summation runs over all TU graphs $H_j$. But the above equation may hold for $m < n$. Consider the following example.
		\begin{example} \label{critical_graphs}
			For the following graphs $K$ and $K^\tau$ we have,
			\begin{equation}
			\begin{split}
			Q_K(\lambda) & = \det(Q(K) - \lambda I) = \lambda^4 - 6\lambda^3 + 9\lambda^2 - 4\lambda \\
			\text{and}, Q_{K^\tau}(\lambda) & = \det(Q(K^\tau) - \lambda I) = \lambda^4 - 6\lambda^3 + 9\lambda^2 - 4\lambda. 
			\end{split}
			\end{equation}
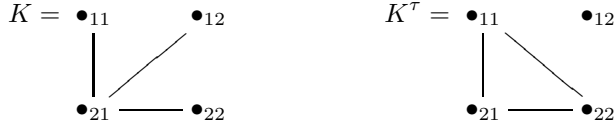
			$$K = \xymatrix{\bullet_{11} \ar@{-}[d] & \bullet_{12} \ar@{-}[dl] \\ \bullet_{21} \ar@{-}[r] & \bullet_{22}} \hspace{2cm} K^\tau = \xymatrix{\bullet_{11} \ar@{-}[d] \ar@{-}[dr] & \bullet_{12} \\ \bullet_{21} \ar@{-}[r] & \bullet_{22}} $$
			Note that, both $K$ and $K^\tau$ have three TU subgraphs with one edges, and three TU subgraphs of two edges. Therefore $p_1 = -6$ and $p_2 = 9$ for both the cases. The TU subgraph with three edges consists is a tree in $K$. Therefore, for $Q_K(\lambda), p_3 = -4^0(1 + 3) = -4$. The TU subgraph with three edges in $K^\tau$ is an odd cycle, such that, for $Q_{K^\tau}(\lambda), p_3 = -4^1(1 + 0) = -4$. Therefore, $K$ and $K^\tau$ are non-isomorphic $Q$-cospectral graphs, such that, $n > m$. This example also shows that two non-isomorphic TU subgraphs may have equal weights.
		\end{example}

	\section{Properties of partial transpose}
	
		Graph theoretic partial transpose was first defined in \cite{wu2006conditions}. In general we consider a graph with $n = p \times q$ vertices. The vertex set is partitioned into $p$ clusters each containing $q$ vertices. In this work, we consider a special case of partial transpose. Let $G$ has even number of vertices, that is $n = 2 \times q$. We can partition the vertex set into clusters as
		\begin{equation} \label{clustering}
		\begin{split} 
		& V(G) = C_1 \cup C_2, ~\text{such that}~ C_1 \cap C_2 = \emptyset, \\
		\text{and}~ & C_1 = \{v_{1, 1}, v_{1, 2}, \dots v_{1, q}\},\\
		& C_2 = \{v_{2, 1}, v_{2, 2}, \dots v_{2, q}\}.
		\end{split} 
		\end{equation}
		The induced subgraph of $G$ generated by the vertex subset $C_1$ and $C_2$ are denoted by $\langle C_1 \rangle_G$ and $\langle C_2 \rangle_G$, respectively. The spanning subgraph of $G$ with edges $\{(u, v): u \in C_1, v \in C_2\}$ is denoted by $\langle C_1 , C_2 \rangle_G$. If there is no confusion with the graph $G$, for simplicity, we drop the suffixes and denote those subgraphs as $\langle C_1 \rangle$, $\langle C_2 \rangle$, and $\langle C_1, C_2 \rangle$, respectively.
		
		\begin{definition}
			The partial transpose of a clustered graph $G$ is denoted by $G^\tau$ obtained by removing all existing edges $(v_{1, i}, v_{2, j})$ from $G$ and adding the corresponding non-existing edges $(v_{1, j}, v_{2, i})$ to $G$, for all $i \neq j$.
		\end{definition}
		
		For instance, consider the graphs $K$ and $K^\tau$ depicted in example 1. We replace the existing edge $(v_{1, 2}, v_{2, 1})$ with $(v_{1, 1}, v_{2, 2})$ to obtain the partial transpose $K^\tau$ of $K$. 	Note that, partial transpose is labelling dependent. Therefore, one graph may produce different graphs after partial transpose. For example the following figure, the graph $G_0$ remains invarient under partial transpose. But, its isomorphic copy $G$ produces a non-isomorphic graph $G^\tau$.
		
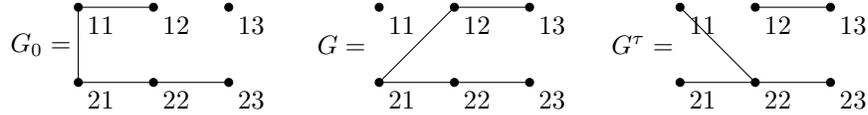
		\begin{tikzpicture}
			\node at (0, 1.5) {};
			\node at (0, -.5) {};
			\node at (-.5, .5) {$G_0 = $};
			\draw[fill] (0, 0) circle [radius= 1.5pt];
			\node[below right] at (0, 0) {$21$};
			\draw[fill] (1, 0) circle [radius= 1.5pt];
			\node[below right] at (1, 0) {$22$};
			\draw[fill] (2, 0) circle [radius= 1.5pt];
			\node[below right] at (2, 0) {$23$};
			\draw[fill] (0, 1) circle [radius= 1.5pt];
			\node[below right] at (0, 1) {$11$};
			\draw[fill] (1, 1) circle [radius= 1.5pt];
			\node[below right] at (1, 1) {$12$};
			\draw[fill] (2, 1) circle [radius= 1.5pt];
			\node[below right] at (2, 1) {$13$};
			\draw (2, 0) -- (0, 0) -- (0, 1) -- (1, 1);
			
			\node at (3.5, .5) {$G = $};
			\draw[fill] (4, 0) circle [radius= 1.5pt];
			\node[below right] at (4, 0) {$21$};
			\draw[fill] (5, 0) circle [radius= 1.5pt];
			\node[below right] at (5, 0) {$22$};
			\draw[fill] (6, 0) circle [radius= 1.5pt];
			\node[below right] at (6, 0) {$23$};
			\draw[fill] (4, 1) circle [radius= 1.5pt];
			\node[below right] at (4, 1) {$11$};
			\draw[fill] (5, 1) circle [radius= 1.5pt];
			\node[below right] at (5, 1) {$12$};
			\draw[fill] (6, 1) circle [radius= 1.5pt];
			\node[below right] at (6, 1) {$13$};
			\draw (6, 1) -- (5, 1) -- (4, 0) -- (6, 0);
			
			\node at (7.5, .5) {$G^\tau = $};
			\draw[fill] (8, 0) circle [radius= 1.5pt];
			\node[below right] at (8, 0) {$21$};
			\draw[fill] (9, 0) circle [radius= 1.5pt];
			\node[below right] at (9, 0) {$22$};
			\draw[fill] (10, 0) circle [radius= 1.5pt];
			\node[below right] at (10, 0) {$23$};
			\draw[fill] (8, 1) circle [radius= 1.5pt];
			\node[below right] at (8, 1) {$11$};
			\draw[fill] (9, 1) circle [radius= 1.5pt];
			\node[below right] at (9, 1) {$12$};
			\draw[fill] (10, 1) circle [radius= 1.5pt];
			\node[below right] at (10, 1) {$13$};
			\draw (10, 1) -- (9, 1);
			\draw (8, 1) -- (9, 0) --(10, 0);
			\draw (8, 0) -- (9, 0);
		\end{tikzpicture}
		
		The arrangement of vertices into clusters, and total number of vertices remains unchanged after partial transpose. It keeps $\langle C_1 \rangle$ and $\langle C_2 \rangle$ unaltered. If degree of a vertex $v_{\mu, i}$ in the graph $G$ be $d(v_{\mu, i})|_G$, then $\sum_{i = 1}^q d(v_{\mu, i})|_G = \sum_{i = 1}^q d(v_{\mu, i})|_{G^\tau}$ for all $\mu$. Changes in the graph is limited within the partially asymmetric edge set, $\mathcal{A} = \{(v_{1, i}, v_{2, j}) \in E(G): i \neq j ~\text{and}~ (v_{1, j}, v_{2, i}) \notin E(G)\} \subset E(\langle C_1, C_2 \rangle)$. 
		
		In this article, we call two isomorphic graphs $G$ and $H$ are equal if the identity mapping acts as the graph isomorphism and we denote $G = H$. A graph $G$ is called partially symmetric if $G = G^\tau$. Clearly for a partially symmetric graph $\mathcal{A} = \emptyset$. A number of partial symmetric graphs are depicted below:
		$$
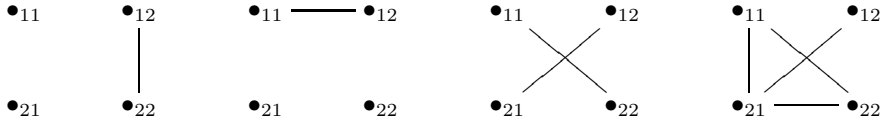
\xymatrix{\bullet_{11} & \bullet_{12} \ar@{-}[d] \\ \bullet_{21} & \bullet_{22}} \hspace{1cm} \xymatrix{\bullet_{11} \ar@{-}[r] & \bullet_{12} \\ \bullet_{21} & \bullet_{22}} \hspace{1cm} \xymatrix{\bullet_{11} \ar@{-}[dr] & \bullet_{12} \ar@{-}[dl] \\ \bullet_{21} & \bullet_{22}} \hspace{1cm} \xymatrix{\bullet_{11} \ar@{-}[d] \ar@{-}[dr] & \bullet_{12} \ar@{-}[dl] \\ \bullet_{21} \ar@{-}[r] & \bullet_{22}}$$
		
		The partial symmetry is different from the usual idea of the symmetry in graph. For instance, the following graph is asymmetric, but, it is partially symmetric with respect to some vertex labelling.
		
		
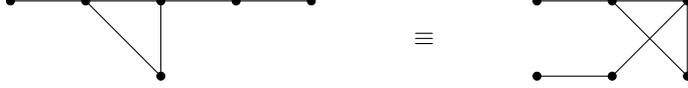
\begin{tikzpicture}
			\node at (0, 1.5) {};
			\node at (0, -.5) {};
			
			\draw[fill] (0, 1) circle [radius= 1.5pt];
			\draw[fill] (1, 1) circle [radius= 1.5pt];
			\draw[fill] (2, 1) circle [radius= 1.5pt];
			\draw[fill] (3, 1) circle [radius= 1.5pt];
			\draw[fill] (4, 1) circle [radius= 1.5pt];
			\draw[fill] (2, 0) circle [radius= 1.5pt];
			\draw (0, 1) -- (4, 1);
			\draw (1, 1) -- (2, 0) -- (2, 1);
			
			\node at (5.5, .5) {$\equiv$};
			
			\draw[fill] (7, 1) circle [radius= 1.5pt];
			\draw[fill] (8, 1) circle [radius= 1.5pt];
			\draw[fill] (9, 1) circle [radius= 1.5pt];
			\draw[fill] (7, 0) circle [radius= 1.5pt];
			\draw[fill] (8, 0) circle [radius= 1.5pt];
			\draw[fill] (9, 0) circle [radius= 1.5pt];
			\draw (7, 1) -- (9, 1) -- (8, 0) -- (7, 0);
			\draw (8, 1) -- (9, 0) -- (9, 1);
		\end{tikzpicture} 
		
		In this following lemma, we mention number of all possible combinations of edges which forms partially symmetric graphs. Recall that, two non-isomorphic graphs may have isomorphic subgraphs. Therefore, given two of these edge combinations may individually generate isomorphic graphs, but they may act as subgraphs of two non-isomorphic graphs. This lemma will help us in calculating number of non-isomorphic graphs with a partially symmetric subgraph.
		\begin{lemma}\label{partially_symmetric_count}
			For any even integer $2q$ there are $2^{\frac{q}{2}(3q - 1)}$ combinations of edges which construct partially symmetric graphs having $2$ clusters with $q$ vertices in each.
		\end{lemma}
		
		\begin{proof}
			We classify the edges of a partially symmetric graph $G$ into the following partitions: $E(\langle C_1 \rangle), E(\langle C_2 \rangle)$, $A = \{(v_{1i}, v_{2i}): i = 1, 2, \dots q\}$, and $B = \{(v_{1i}, v_{2j}), (v_{1j}, v_{2i}): i = 1, 2, \dots q ~\text{and}~ i \neq j\}$. Note that all these edge sets remain invariant under partial transpose.
			
			As $C_1$ has $q$ nodes, total number of possible edges in $\langle C_1 \rangle$ is $^qC_2 = \frac{q(q - 1)}{2}$. The number of all possible combinations of edges in $E(\langle C_1 \rangle)$ is $2^{\frac{q(q - 1)}{2}}$. Similarly, $E(\langle C_1 \rangle)$ also has $2^{\frac{q(q - 1)}{2}}$ combinations of edges.
			
			Note that, number of all possible edges in class $A$ is $q$. In a partially symmetric edge combination any of them may be selected on not. Therefore, possible combinations of edges in class $A$ is $2^q$.
			
			Edges in class $B$ appears in a pair $(v_{1i}, v_{2j}), (v_{1j}, v_{2i})$. Two vertices with suffixes $i$ and $j$ from $q$ vertices can be selected in $^qC_2 = \frac{q(q - 1)}{2}$ ways. Total number of possible combinations of edges in class $B$ is $2^{\frac{q(q - 1)}{2}}$.
			
			Therefore, all possible combinations of edges which forms a partially symmetric graph is $2^{\frac{q(q - 1)}{2}}2^{\frac{q(q - 1)}{2}}2^q2^{\frac{q(q - 1)}{2}} = 2^{\frac{q}{2}(3q - 1)}$.
		\end{proof}
		
		We end up this section with the following example which will clarify the difference between partial transpose and Godsil-McKay switching \cite{godsil1982constructing, van2003graphs}. 
		
		\begin{example}
			Consider the graph $G$ with $8$ vertices. To perform Godsil-McKay switching we arrange the vertex set into two clusters $C$ and $D$. Vertices in $D$ is either connected to all the vertices, or half of the vertices, or no vertex of $C$ \cite{godsil1982constructing}. The resulting graph $G^{GM}$ and $G$ are depicted below:
			
			
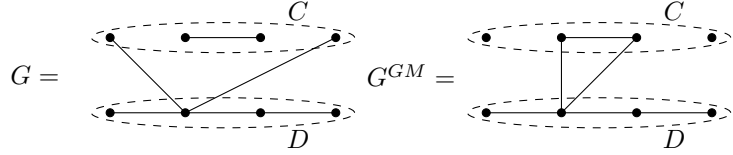
\begin{tikzpicture}
				\node at (0, 1.5) {};
				\node at (0, -.5) {};
				\node at (-1, .5) {$G = $};
				\draw[fill] (0, 0) circle [radius = 1.5pt];
				\draw[fill] (1, 0) circle [radius = 1.5pt];
				\draw[fill] (2, 0) circle [radius = 1.5pt];
				\draw[fill] (3, 0) circle [radius = 1.5pt];
				\draw[fill] (0, 1) circle [radius = 1.5pt];
				\draw[fill] (1, 1) circle [radius = 1.5pt];
				\draw[fill] (2, 1) circle [radius = 1.5pt];
				\draw[fill] (3, 1) circle [radius = 1.5pt];
				\draw (0, 0) -- (1, 0) -- (2, 0) -- (3, 0);
				\draw (1, 1) -- (2, 1);
				\draw (0, 1) -- (1, 0) -- (3, 1);
				\draw[dashed] (1.5, 0) ellipse (1.75cm and .2cm);
				\node at (2.5, -.35) {$D$};
				\draw[dashed] (1.5, 1) ellipse (1.75cm and .2cm);
				\node at (2.5, 1.35) {$C$};
				
				\node at (4, .5) {$G^{GM} = $};
				\draw[fill] (5, 0) circle [radius = 1.5pt];
				\draw[fill] (6, 0) circle [radius = 1.5pt];
				\draw[fill] (7, 0) circle [radius = 1.5pt];
				\draw[fill] (8, 0) circle [radius = 1.5pt];
				\draw[fill] (5, 1) circle [radius = 1.5pt];
				\draw[fill] (6, 1) circle [radius = 1.5pt];
				\draw[fill] (7, 1) circle [radius = 1.5pt];
				\draw[fill] (8, 1) circle [radius = 1.5pt];
				\draw (5, 0) -- (6, 0) -- (7, 0) -- (8, 0);
				\draw (6, 1) -- (7, 1);
				\draw (6, 1) -- (6, 0) -- (7, 1);
				\draw[dashed] (6.5, 0) ellipse (1.75cm and .2cm);
				\node at (7.5, -.35) {$D$};
				\draw[dashed] (6.5, 1) ellipse (1.75cm and .2cm);
				\node at (7.5, 1.35) {$C$};				
			\end{tikzpicture}
			
			Now we perform partial transpose on $G$ taking $C_1 = D$ and $C_2 = C$. 
			
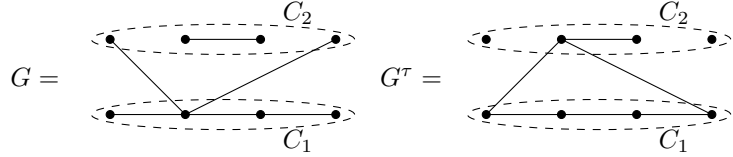
			\begin{tikzpicture}
				\node at (0, 1.5) {};
				\node at (0, -.5) {};
				\node at (-1, .5) {$G = $};
				\draw[fill] (0, 0) circle [radius = 1.5pt];
				\draw[fill] (1, 0) circle [radius = 1.5pt];
				\draw[fill] (2, 0) circle [radius = 1.5pt];
				\draw[fill] (3, 0) circle [radius = 1.5pt];
				\draw[fill] (0, 1) circle [radius = 1.5pt];
				\draw[fill] (1, 1) circle [radius = 1.5pt];
				\draw[fill] (2, 1) circle [radius = 1.5pt];
				\draw[fill] (3, 1) circle [radius = 1.5pt];
				\draw (0, 0) -- (1, 0) -- (2, 0) -- (3, 0);
				\draw (1, 1) -- (2, 1);
				\draw (0, 1) -- (1, 0) -- (3, 1);
				\draw[dashed] (1.5, 0) ellipse (1.75cm and .2cm);
				\node at (2.5, -.35) {$C_1$};
				\draw[dashed] (1.5, 1) ellipse (1.75cm and .2cm);
				\node at (2.5, 1.35) {$C_2$};
				
				\node at (4, .5) {$G^\tau = $};
				\draw[fill] (5, 0) circle [radius = 1.5pt];
				\draw[fill] (6, 0) circle [radius = 1.5pt];
				\draw[fill] (7, 0) circle [radius = 1.5pt];
				\draw[fill] (8, 0) circle [radius = 1.5pt];
				\draw[fill] (5, 1) circle [radius = 1.5pt];
				\draw[fill] (6, 1) circle [radius = 1.5pt];
				\draw[fill] (7, 1) circle [radius = 1.5pt];
				\draw[fill] (8, 1) circle [radius = 1.5pt];
				\draw (5, 0) -- (6, 0) -- (7, 0) -- (8, 0);
				\draw (6, 1) -- (7, 1);
				\draw (5, 0) -- (6, 1) -- (8, 0);
				\draw[dashed] (6.5, 0) ellipse (1.75cm and .2cm);
				\node at (7.5, -.35) {$C_1$};
				\draw[dashed] (6.5, 1) ellipse (1.75cm and .2cm);
				\node at (7.5, 1.35) {$C_2$};		
			\end{tikzpicture}
			
			Clearly, $G^\tau$ is non-isomorphic to $G^{GM}$.
		\end{example}

	\section{Number of non-isomorphic graphs which are $Q$-cospectral to their partial transpose}
				
		In the example \ref{critical_graphs}, we have seen that $K$ and $K^\tau$ are $Q$-cospectral. Also, we have mentioned that $K^\tau$ is the partial transpose of $K$. In fact, $K$ is the smallest graph which is non-isomorphic and $Q$-cospectral to its partial transpose. But, not all graphs are $Q$-cospectral to their partial transpose, for instance, consider the following graphs:
		$$G = 
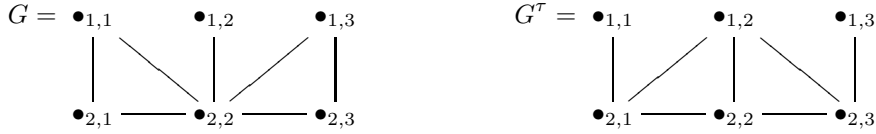
\xymatrix{\bullet_{1,1} \ar@{-}[d] \ar@{-}[dr] & \bullet_{1,2} \ar@{-}[d] & \bullet_{1,3} \ar@{-}[d] \ar@{-}[dl] \\ \bullet_{2,1} \ar@{-}[r]& \bullet_{2,2} & \bullet_{2,3} \ar@{-}[l]} \hspace{2cm} G^\tau = \xymatrix{\bullet_{1,1} \ar@{-}[d] & \bullet_{1,2} \ar@{-}[d] \ar@{-}[dr] \ar@{-}[dl] & \bullet_{1,3} \ar@{-}[d] \\ \bullet_{2,1} \ar@{-}[r] & \bullet_{2,2} & \bullet_{2,3} \ar@{-}[l]}$$
		It is easy to calculate that $Q$-spectra of $G$ and $G^\tau$ are $\{0.6277, 1, 1, 2, 3, 6.3723\}$ and $\{0.3542, 0.5858, 2, 2, 3.4142, 5.6458\}$. Also, there is no vertex labelling, such that, any of the following two $Q$-cospectral graphs are partial transpose of another: 
		$$
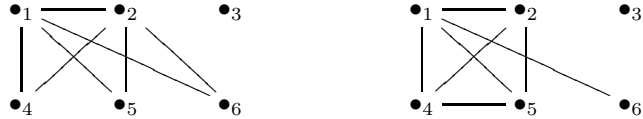
\xymatrix{\bullet_1 \ar@{-}[d] \ar@{-}[r] \ar@{-}[dr] \ar@{-}[drr] & \bullet_2 \ar@{-}[d] \ar@{-}[dr] \ar@{-}[dl] & \bullet_3 \\ \bullet_4 & \bullet_5 & \bullet_6} \hspace{2 cm} \xymatrix{\bullet_1 \ar@{-}[d] \ar@{-}[r] \ar@{-}[dr] \ar@{-}[drr] & \bullet_2 \ar@{-}[d] \ar@{-}[dl] & \bullet_3 \\ \bullet_4 \ar@{-}[r] & \bullet_5 & \bullet_6}$$
		We can check this assertion by considering every vertex labellings on the above graphs using a suitable computer algebra system. 
		
		There are big families of graphs which are $Q$-cospectral to their partial transpose. The following table provides number of graphs which are non-isomorphic, and $Q$-cospectral to their partial transpose. We use Networkx library \cite{schult2008exploring} for generating the following computational data and all examples which are included in this article.
		\newpage
		\begin{longtable}{|p{.15 \textwidth}|p{.1 \textwidth}| p{.2 \textwidth}| p{.26 \textwidth}| p{.1 \textwidth}|}
%		\begin{longtable}{|c|c| c| c| c|}
			\hline
			Number of vertices & Number of edges & Number of non-isomorphic $Q$-cospectrals & Number of non-isomorphic $Q$-cospectrals to PT & Ratio\\
			\hline
			4 & 3 & 2 & 2 & 1\\
			\hline 
			5 & 3 & 2 & 2 & 1 \\
			\cline{2 - 5}
			& 7 & 2 & 0 & 0 \\
			\hline 
			6 & 3 & 2 & 2 & 1 \\
			\cline{2 - 5}
			& 4 & 2 & 2 & 1 \\
			\cline{2 - 5}
			& 7 & 4 & 2 & .5 \\
			\hline
			7 & 3 & 2 & 2 & 1\\
			\cline{2 - 5}
			& 4 & 2 & 2 & 1 \\
			\cline{2 - 5}
			& 5 & 2 & 2 & 1 \\
			\cline{2 - 5}
			& 6 & 2 & 0 & 0 \\
			\cline{2 - 5}
			& 7 & 6 & 4 & $.667$ \\
			\cline{2 - 5}
			& 8 & 12 & 8 & $.667$ \\
			\cline{2 - 5}
			& 9 & 14 & 10 & $.714$ \\
			\cline{2 - 5}
			& 10 & 14 & 10 & $.714$ \\
			\cline{2 - 5}
			& 11 & 14 & 12 & $.857$ \\
			\cline{2 - 5}
			& 12 & 12 & 12 & 1 \\
			\cline{2 - 5}
			& 13 & 12 & 10 & $.833$ \\
			\cline{2 - 5}
			& 14 & 6 & 2 & $.333$ \\
			\cline{2 - 5}
			& 15 & 2 & 0 & 0 \\
			\cline{2 - 5}
			& 16 & 2 & 0 & 0 \\
			\cline{2 - 5}
			& 17 & 2 & 0 & 0 \\
			\hline 
			8 & 3 & 2 & 2 & 1 \\
			\cline{2 - 5}
			& 4 & 2 & 2 & 1 \\
			\cline{2 - 5}
			& 5 & 4 & 4 & 0 \\
			\cline{2 - 5}
			& 6 & 12 & 8 & $.667$ \\
			\cline{2 - 5}
			& 7 & 20 & 14 & $.7$  \\
			\cline{2 - 5}
			& 8 & 38 & 26 & .684\\
			\cline{2 - 5}
			& 9 & 58 & 42 & .724\\
			\hline
		\end{longtable}	
		This statistics suggests that $75\%$ among the non-isomorphic $Q$-cospectral graphs with $6$ vertices can be determined by partial transpose. For $7$ vertex graphs this ration is $71.15\%$. For graphs with $8$ vertices the ratio is $71.01\%$ which is computed up to our limitation. Therefore, a large class of graphs are non-isomorphic and $Q$-cospectral to their partial transpose. These graphs follows a number of patterns which we shall discuss in the following section.

	\section{When $G$ and $G^\tau$ are non-isomorphic and $Q$-cospectral?}
		
		Two $Q$-cospectral graphs $G$ and $G^\tau$ have equal $Q$-polynomials, that is, the coefficients of $Q_G(\lambda)$, and $Q_{G^\tau}(\lambda)$ are equal. Recall that, the coefficients of $Q_G(\lambda)$ depend on TU subgraphs of $G$. Let $\mathcal{U}_j(G)$ be the set of all TU subgraphs of $j$ edges. Two sets of TU subgraphs $\mathcal{U}_j(G)$ and $\mathcal{U}_j(G^\tau)$ are comparable if 
		\begin{equation}
			\sum_{H \in \mathcal{U}_j(G)} W(H) = \sum_{H \in \mathcal{U}_j(G^\tau)} W(H),
		\end{equation}
		where $W(H)$ are determined by the equation (\ref{TU_weight}). Now equation (\ref{char_coeff}) suggests that, if $G$ and $G^\tau$ are $Q$-cospectral the sets of their TU subgraphs are comparable for all $j = 1, 2, \dots m$. We call two graphs $G$ and $H$ are comparable if $\mathcal{U}_j(G)$ and $\mathcal{U}_j(H)$ are comparable for all $j$. As an example two tree with equal number of edges are comparable. Similarly, two circles of equal lengths are comparable. In example \ref{critical_graphs} we have already seen that the TU subgraphs of $G$ and $G^\tau$ have equal weights but they are not isomorphic. Here, we find conditions on graphs which keep $\mathcal{U}_j(G)$ and $\mathcal{U}_j(G^\tau)$ comparable for all $j$. 
		
	\begin{theorem}\label{theorem1}
		Let the subgraphs $\langle C_1 \rangle_{G_0}$ and $\langle C_2 \rangle_{G_0}$ of the graph $G_0$ be two $q$-cycles as well as $\langle C_1, C_2\rangle_{G_0}$ be an empty graph. Also, let $v_{1, i}$ and $v_{1, j}$ be two non-adjacent vertices of $G_0$. We add the edges $(v_{1, i}, v_{1, j}), (v_{1, i}, v_{2, i})$ and $(v_{1, i}, v_{2, j})$ with $G_0$. The new graph $G$ is non-isomorphic and Q-cospectral to its partial transpose.
	\end{theorem}

	\begin{proof}
		Clearly, $G = G_0 \cup \{(v_{1, i}, v_{1, j}), (v_{1, i}, v_{2, i}), (v_{1, i}, v_{2, j})\}$. The set of all cycles in $G$ consists of two cycles of $G_0$. Call them $\delta_1$ and $\delta_2$. The following new cycles are generated by additional three edges and their incidence with existing edges in $G_0$:
		\begin{enumerate}
			\item 
				$\delta_3 = (v_{1, i}, v_{1, i+1}, v_{1, i + 2}, \dots v_{1, j})$,
			\item
				$\delta_4 = (v_{1, 1}, v_{1, 2}, \dots v_{1, i}, \dots v_{1, j}, v_{1, j+1}, \dots v_{1, q})$,
			\item
				$\delta_5 = (v_{1, i}, v_{2, i}, v_{2, i+ 1}, v_{2, i + 2}, \dots v_{2, j})$, and
			\item
				$\delta_6 = (v_{2, 1}, v_{2,2}, \dots v_{2, i}, v_{1, i}, v_{2, j}, v_{2, j + 1}, \dots v_{2, q})$.
		\end{enumerate}
		Note that, $\langle C_1, C_2 \rangle_G$ contains only two edges which are $(v_{1, i}, v_{2, i})$ and $(v_{1, i}, v_{2, j})$. Partial transpose replace $(v_{1, i}, v_{2, j})$ with $(v_{1, j}, v_{2, i})$. The cycles $\delta_1, \delta_2, \delta_3$ and $\delta_4$ remain invariant under partial transpose on $G$. Therefore, their TU subgrphs are isomorphic in $G$ and $G^\tau$ and have equal contribution in $Q_G(\lambda)$, and $Q_{G^\tau}(\lambda)$.
		
		Now $\delta_5$ in $G$ is replaced by $\delta_5' = (v_{1, i}, v_{2, i}, v_{1, j}, v_{1, j-1}, \dots  v_{1, i+1})$ in $G^\tau$. They have equal length and equal contribution in the characteristic coefficients. The circle $\delta_6$ in $G$ and its counterpart $\delta_6' = (v_{1, 1}, v_{1, 2}, \dots v_{1, i}, v_{2,i}, v_{1, j}, v_{1, j+1}, \dots v_{1, q})$ in $G^\tau$ have equal lengths, $|\delta_6| = |\delta_6'| = q - (j - i) + 2$. If $(v_{2, k}, v_{2, k + 1}) \in \delta_6 \cap c_2$ in $G$ then $(v_{1, k}, v_{1, k + 1}) \in \delta_6' \cap c_1$ in $G^\tau$. A TU subgraph containing more than $|\delta_6|$ edges contains edges from $\delta_1$ in $G$. The role of $\delta_1$ in $G$ is replaced by the edges of $\delta_2$ in $G^\tau$. We have assumed that $\delta_1$ and $\delta_2$ have equal length. Therefore, replacement of $\delta_6$ in $G^\tau$ does not make any difference in the characteristic coefficients.
		
		The new edges $K = \{(v_{1, i}, v_{1, j}), (v_{1, i}, v_{2, i}), (v_{1, i}, v_{2, j})\}$ forms a tree in $G$. It is replaced by an uni-cyclic TU subgraph $K^\tau = (v_{1, i}, v_{1, j}, v_{2, i})$ in $G^\tau$. They have equal contribution in $Q_G(\lambda)$ and $Q_{G^\tau}(\lambda)$ that we have seen in example \ref{critical_graphs}.
		
		Therefore, all the TU subgraphs of $G$ and $G^\tau$ are comparable as well as they form equal characteristic polynomials. Hence, $G$ is $Q$-cospectral to its partial transpose.
	\end{proof}
	
	Note that, if $v_{1,i}$ and $v_{1, j}$ are adjacent in $G_0$ we may construct $G = G_0 \cup \{(v_{1, i}, v_{2, i}), (v_{1, i}, v_{2, j})\}$. We can easily prove that $G$ ad $G^\tau$ are isomorphic and Q-cospectral.
	
	Given any integer $q$ there is only one $q$-cycle which is considered as $\langle C_1 \rangle$, and $\langle C_2 \rangle$. For any vertex $v_{1, i} \in C_1$ there are $(q - 2)$ non-adjacent vertices which are possible choice of $v_{1, j}$. Also, we can choose $v_{1,i}$ in $q$ ways, but it will generate isomorphic families of graphs. We can check it by considering two graphs generated by choosing $v_{1, 1}$ and $v_{1, i}$. Therefore, there are $2^{(q - 2)}$ non-isomorphic graphs which are non-isomorphic and $Q$-cospectral to their partial transposes.
	
	\begin{example}
		Consider the following graph $G$ with its partial transpose $G^\tau$:
		
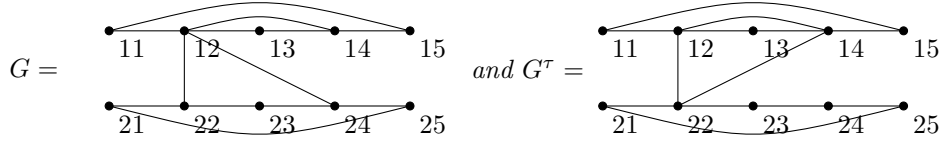
		\begin{tikzpicture}[scale = 1]
			\node at (-1, .5) {$G = $};
			\draw[fill] (0, 0) circle [radius= 1.5pt];
			\node[below right] at (0, 0) {$21$};
			\draw[fill] (1, 0) circle [radius= 1.5pt];
			\node[below right] at (1, 0) {$22$};
			\draw[fill] (2, 0) circle [radius= 1.5pt];
			\node[below right] at (2, 0) {$23$};
			\draw[fill] (3, 0) circle [radius= 1.5pt];
			\node[below right] at (3, 0) {$24$};
			\draw[fill] (4, 0) circle [radius= 1.5pt];
			\node[below right] at (4, 0) {$25$};
			\draw[fill] (0, 1) circle [radius= 1.5pt];
			\node[below right] at (0, 1) {$11$};
			\draw[fill] (1, 1) circle [radius= 1.5pt];
			\node[below right] at (1, 1) {$12$};
			\draw[fill] (2, 1) circle [radius= 1.5pt];
			\node[below right] at (2, 1) {$13$};
			\draw[fill] (3, 1) circle [radius= 1.5pt];
			\node[below right] at (3, 1) {$14$};
			\draw[fill] (4, 1) circle [radius= 1.5pt];
			\node[below right] at (4, 1) {$15$};
			\draw (0, 0) -- (1, 0) -- (2, 0) -- (3, 0) -- (4, 0);
			\draw (0, 1) -- (1, 1) -- (2, 1) -- (3, 1) -- (4, 1);
			\draw (1, 0) -- (1, 1) -- (3, 0);
			\draw (0,0) .. controls (2, -.5) .. (4, 0);
			\draw (0,1) .. controls (2, 1.5) .. (4, 1);
			\draw (1,1) .. controls (2, 1.25) .. (3, 1);
		\end{tikzpicture}
		\begin{tikzpicture}[scale = 1]
			\node at (-1, .5) {and $G^\tau = $};
			\draw[fill] (0, 0) circle [radius= 1.5pt];
			\node[below right] at (0, 0) {$21$};
			\draw[fill] (1, 0) circle [radius= 1.5pt];
			\node[below right] at (1, 0) {$22$};
			\draw[fill] (2, 0) circle [radius= 1.5pt];
			\node[below right] at (2, 0) {$23$};
			\draw[fill] (3, 0) circle [radius= 1.5pt];
			\node[below right] at (3, 0) {$24$};
			\draw[fill] (4, 0) circle [radius= 1.5pt];
			\node[below right] at (4, 0) {$25$};
			\draw[fill] (0, 1) circle [radius= 1.5pt];
			\node[below right] at (0, 1) {$11$};
			\draw[fill] (1, 1) circle [radius= 1.5pt];
			\node[below right] at (1, 1) {$12$};
			\draw[fill] (2, 1) circle [radius= 1.5pt];
			\node[below right] at (2, 1) {$13$};
			\draw[fill] (3, 1) circle [radius= 1.5pt];
			\node[below right] at (3, 1) {$14$};
			\draw[fill] (4, 1) circle [radius= 1.5pt];
			\node[below right] at (4, 1) {$15$};
			\draw (0, 0) -- (1, 0) -- (2, 0) -- (3, 0) -- (4, 0);
			\draw (0, 1) -- (1, 1) -- (2, 1) -- (3, 1) -- (4, 1);
			\draw (1, 0) -- (1, 1);
			\draw (1, 0) -- (3, 1);
			\draw (0,0) .. controls (2, -.5) .. (4, 0);
			\draw (0,1) .. controls (2, 1.5) .. (4, 1);
			\draw (1,1) .. controls (2, 1.25) .. (3, 1);
		\end{tikzpicture}
		
		Here, $q = 5$. Cycles in $G$ are:
		\begin{enumerate}
			\item 
				$\delta_1 = (v_{11}, v_{12}, \dots v_{15})$ with $|\delta_1| = 5$,
			\item 
				$\delta_2 = (v_{21}, v_{22}, \dots v_{25})$ with $|\delta_2| = 5$,
			\item 
				$\delta_3 = (v_{12}, v_{13}, v_{14})$ with $|\delta_3| = 3$,
			\item
				$\delta_4 = (v_{11}, v_{12}, v_{14}, v_{15})$ with $|\delta_4| = 4$,
			\item
				$\delta_5 = (v_{12}, v_{22}, v_{23}, v_{24})$ with $|\delta_5| = 4$,
			\item
				$\delta_6 = (v_{21}, v_{22}, v_{12}, v_{24}, v_{25})$ with $|\delta_6| = 5$.
		\end{enumerate}
		There are four uni-cyclic graphs which are $\delta_1, \delta_2, \delta_3$, and $\delta_6$. Here, $\delta_1, \delta_2, \delta_3$ remains invariant under partial transpose. Also, $\delta_6$ is replaced by $\delta_6' = (v_{11}, v_{12}, $ $v_{22}, v_{14}, v_{15})$. There is only one change among trees. The subgraph $K$ in $G$ is transformed to the odd unicyclic graph $K^\tau$ in $G^\tau$. Therefore, all TU subgraphs of $G$ and $G^\tau$ are comparable. Therefore, $G$ and $G^\tau$ are cospectral. The subgraphs $K$ and $K^\tau$ make the graphs $G$ and $G^\tau$ non-isomorphic.
	\end{example}

	\begin{corollary}\label{corollary1}
		Let the subgraphs $\langle C_1 \rangle_{G_0}$ and $\langle C_2 \rangle_{G_0}$ of the graph $G_0$ be two $q$-cycles as well as $\langle C_1, C_2\rangle_{G_0}$ be an empty graph. We add the edges $(v_{1, i}, v_{2, i})$ and $(v_{1, i}, v_{2, i + 1})$ as well as remove the edge $(v_{2, i}, v_{2, i+1})$. The new graph $G$ is non-isomorphic and $Q$-cospectral to its partial transpose.
	\end{corollary}
	
	\begin{proof}
		Verification of $Q$-cospectrality of $G$ and $G^\tau$ is similar to the theorem \ref{theorem1}. Non-existence of the edge $(v_{2, i}, v_{2, i+1})$ and alteration of $(v_{1, i}, v_{2, i + 1})$ during partial transpose makes $G$ non-isomorphic to $G^\tau$.
	\end{proof}
	
	Here if we do not remove $(v_{2, i}, v_{2, i+1})$, then $G$ is isomorphic and $Q$-cospectral to its partial transpose. One can check it by keeping edge $(v_{2,1}, v_{2,2})$ in the example below.
	
	We can select a vertex $v_{1, i}$ from the vertices of $C_1$ in $q$ ways. For every such choice we may construct a graph $G$. We can check that all these graphs will be isomorphic to each other. Therefore, for any integer $q$ there is only $1$ graph $G$ constructed with this theorem which is non-isomorphic and $Q$-cospectral to its partial transpose.
	
	\begin{example}
		In the figure below\\		
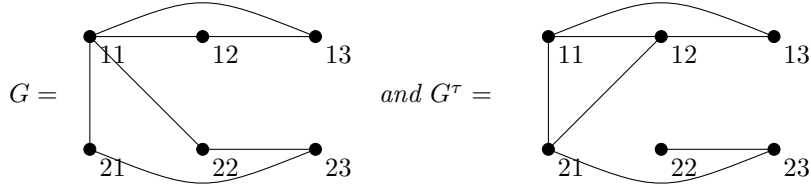
		\begin{tikzpicture}[scale = 1.5]
		\node at (-.5, .5) {$G = $};
		\draw[fill] (0, 0) circle [radius= 1.5pt];
		\node[below right] at (0, 0) {$21$};
		\draw[fill] (1, 0) circle [radius= 1.5pt];
		\node[below right] at (1, 0) {$22$};
		\draw[fill] (2, 0) circle [radius= 1.5pt];
		\node[below right] at (2, 0) {$23$};
		\draw[fill] (0, 1) circle [radius= 1.5pt];
		\node[below right] at (0, 1) {$11$};
		\draw[fill] (1, 1) circle [radius= 1.5pt];
		\node[below right] at (1, 1) {$12$};
		\draw[fill] (2, 1) circle [radius= 1.5pt];
		\node[below right] at (2, 1) {$13$};
		\draw (0, 1) -- (1, 1) -- (2, 1);
		\draw (0, 1) -- (0, 0);
		\draw (0, 1) -- (1, 0) -- (2, 0);
		\draw (0, 0) .. controls (1, -.4) .. (2, 0);
		\draw (0, 1) .. controls (1, 1.4) .. (2, 1);
		\end{tikzpicture}
		\begin{tikzpicture}[scale = 1.5]
		\node at (-1, .5) {and $G^\tau = $};
		\draw[fill] (0, 0) circle [radius= 1.5pt];
		\node[below right] at (0, 0) {$21$};
		\draw[fill] (1, 0) circle [radius= 1.5pt];
		\node[below right] at (1, 0) {$22$};
		\draw[fill] (2, 0) circle [radius= 1.5pt];
		\node[below right] at (2, 0) {$23$};
		\draw[fill] (0, 1) circle [radius= 1.5pt];
		\node[below right] at (0, 1) {$11$};
		\draw[fill] (1, 1) circle [radius= 1.5pt];
		\node[below right] at (1, 1) {$12$};
		\draw[fill] (2, 1) circle [radius= 1.5pt];
		\node[below right] at (2, 1) {$13$};
		\draw (0, 1) -- (1, 1) -- (2, 1);
		\draw (0, 1) -- (0, 0);
		\draw (1, 0) -- (2, 0);
		\draw (0, 0) -- (1, 1);
		\draw (0, 0) .. controls (1, -.4) .. (2, 0);
		\draw (0, 1) .. controls (1, 1.4) .. (2, 1);
		\end{tikzpicture}\\
		are $Q$-cospectral, non-isomorphic graphs. The graph $G$ is constructed by the above theorem.
	\end{example}

	\begin{corollary}\label{corollary2} 
		Let the subgraphs $\langle C_1 \rangle_{G_0}$ be a $q$-cycle and $\langle C_2 \rangle_{G_0}$ be a  path graph of length $q$ as well as $\langle C_1, C_2\rangle_{G_0}$ is an empty graph. Construct a new graph $G$ by adding $(v_{11}, v_{1q}), (v_{11}, v_{21})$ and $(v_{11}, v_{2q})$. In addition, any edge of the form $(v_{1k}, v_{2k})$ can be included in $G$. The new graph $G$ is $Q$-cospectral to its partial transpose.
	\end{corollary}
	
	\begin{proof}
		Proof is similar to theorem \ref{theorem1}. 
	\end{proof}
	\begin{example}
		In the graph $G$ we have taken a $5$-cycle as $\langle C_1 \rangle$ and a path of length $5$ as $\langle C_2 \rangle$. We have added the edges $(v_{11}, v_{25})$ and $(v_{11}, v_{12})$ for generating non-isomorphic graphs under partial transpose. Also, we have added $(v_{12}, v_{22}), (v_{13}, v_{23}), (v_{15}, v_{25})$ which remains unchanged under partial transpose. The resultant graph $G$ and its partial transpose are:\\
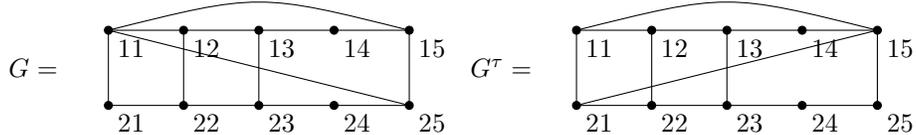
		\begin{tikzpicture}
		\node at (0, .5) {$G = $};
		\draw[fill] (1, 1) circle [radius = 1.5pt];
		\node[below right] at (1,1) {$11$};
		\draw[fill] (2, 1) circle [radius = 1.5pt];
		\node[below right] at (2,1) {$12$};
		\draw[fill] (3, 1) circle [radius = 1.5pt];
		\node[below right] at (3,1) {$13$};
		\draw[fill] (4, 1) circle [radius = 1.5pt];
		\node[below right] at (4,1) {$14$};
		\draw[fill] (5, 1) circle [radius = 1.5pt];
		\node[below right] at (5,1) {$15$};
		\draw[fill] (1, 0) circle [radius = 1.5pt];
		\node[below right] at (1,0) {$21$};
		\draw[fill] (2, 0) circle [radius = 1.5pt];
		\node[below right] at (2,0) {$22$};
		\draw[fill] (3, 0) circle [radius = 1.5pt];
		\node[below right] at (3,0) {$23$};
		\draw[fill] (4, 0) circle [radius = 1.5pt];
		\node[below right] at (4,0) {$24$};
		\draw[fill] (5, 0) circle [radius = 1.5pt];
		\node[below right] at (5,0) {$25$};
		\draw (1, 0) -- (2, 0) -- (3, 0) -- (4, 0) -- (5, 0) -- (5, 1) -- (4, 1) -- (3, 1) -- (2, 1) -- (1, 1) -- (1, 0);
		\draw (1, 1) -- (5, 0);
		\draw (2, 1) -- (2, 0);
		\draw (3, 1) -- (3, 0);
		\draw (1,1) .. controls (3, 1.5) .. (5, 1); 
		\end{tikzpicture}
		\begin{tikzpicture}
		\node at (0, .5) {$G^\tau = $};
		\draw[fill] (1, 1) circle [radius = 1.5pt];
		\node[below right] at (1,1) {$11$};
		\draw[fill] (2, 1) circle [radius = 1.5pt];
		\node[below right] at (2,1) {$12$};
		\draw[fill] (3, 1) circle [radius = 1.5pt];
		\node[below right] at (3,1) {$13$};
		\draw[fill] (4, 1) circle [radius = 1.5pt];
		\node[below right] at (4,1) {$14$};
		\draw[fill] (5, 1) circle [radius = 1.5pt];
		\node[below right] at (5,1) {$15$};
		\draw[fill] (1, 0) circle [radius = 1.5pt];
		\node[below right] at (1,0) {$21$};
		\draw[fill] (2, 0) circle [radius = 1.5pt];
		\node[below right] at (2,0) {$22$};
		\draw[fill] (3, 0) circle [radius = 1.5pt];
		\node[below right] at (3,0) {$23$};
		\draw[fill] (4, 0) circle [radius = 1.5pt];
		\node[below right] at (4,0) {$24$};
		\draw[fill] (5, 0) circle [radius = 1.5pt];
		\node[below right] at (5,0) {$25$};
		\draw (1, 0) -- (2, 0) -- (3, 0) -- (4, 0) -- (5, 0) -- (5, 1) -- (4, 1) -- (3, 1) -- (2, 1) -- (1, 1) -- (1, 0);
		\draw (1, 0) -- (5, 1);
		\draw (2, 1) -- (2, 0);
		\draw (3, 1) -- (3, 0);
		\draw (1,1) .. controls (3, 1.5) .. (5, 1); 
		\end{tikzpicture}\\
		It can be easily verified that $G$ and $G^\tau$ are $Q$-cospectral.
	\end{example}

	\section{Bigger families of non-isomorphic $Q$-cospectral graphs}
	
	In the last section, we have mentioned structures of graphs which are non-isomorphic and $Q$-cospectral to their partial transpose. Given any graph of this kind there are infinitely many graphs of bigger size  which are also non-isomorphic and $Q$-cospectral to their partial transpose. Now we shall state a number of procedures for constructing these graphs.
	
	\begin{procedure}\label{procedure1}
		Let $G$ be $Q$-cospectral to its partial transpose $G^\tau$. Construct a new graph $G_1 = G \cup G'$ such that $G'$ is isomorphic to its partial transpose. Then, $G_1$ is $Q$-cospectral to $G_1^\tau$.
	\end{procedure}		
	
	\begin{proof}
		As $G_1 = G \cup G'$, $G_1^\tau = G^\tau \cup (G')^\tau = G^\tau \cup G'$, as $G'$ is isomorphic to its partial transpose. Now, $Q_{G_1}(\lambda) = Q_{G \cup G'}(\lambda) = Q_G(\lambda)Q_{G'}(\lambda)$. Also, $Q_{G_1^\tau}(\lambda) =  Q_{G^\tau}(\lambda)Q_{G'}(\lambda)$. We assumed that $G$ and $G^\tau$ are cospectral. Hence, $Q_{G}(\lambda) = Q_{G^\tau}(\lambda)$. Combining these all we get, $Q_{G_1}(\lambda) = Q_{G_1^\tau}(\lambda)$. Therefore, $G_1$ and $G_1^\tau$ are cospectral.
	\end{proof}
	
	If in the above theorem $G$ is non-isomorphic and $Q$-cospectral to $G^\tau$ then the resultant graph $G_1$ is also non-isomorphic and $Q$-cospectral to $G_1^\tau$. Note that, using the above result arbitrary large non-isomorphc and $Q$ cospectral graphs can be generated. For simplicity, one may consider any partially symmetric graph as $G'$.
	
	\begin{example}
		Consider the graph $G = K$ depicted in example \ref{critical_graphs}, for simplicity. Add vertices $v_{13}, v_{23}$ and an edge $(v_{13}, v_{23})$ to construct new graph $G_1$. Note that, here $G'$ consists of a single edge $(v_{13}, v_{23})$ which is a partial symmetric.  
		$$G_1 = 
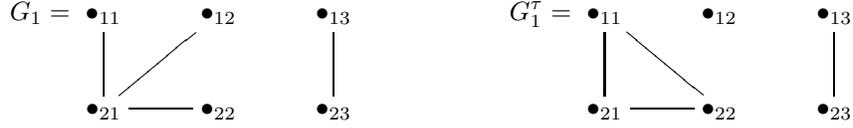
\xymatrix{\bullet_{11} \ar@{-}[d] & \bullet_{12} \ar@{-}[dl] & \bullet_{13} \ar@{-}[d] \\ \bullet_{21} \ar@{-}[r] & \bullet_{22} & \bullet_{23} } \hspace{2 cm} G_1^\tau = \xymatrix{\bullet_{11} \ar@{-}[d] \ar@{-}[dr] & \bullet_{12} & \bullet_{13} \ar@{-}[d] \\ \bullet_{21} \ar@{-}[r] & \bullet_{22} & \bullet_{23} }$$
		We can easily verify that $G_1$ and $G_1^\tau$, depicted above, are non-isomorphic, and $Q$-cospectral graphs.
	\end{example}
	
	The above result can be visualised in terms of matrices. Let $A$ and $B$ be the signless Laplacian matrices of two $Q$-cospectral graphs $G$ and $G^\tau$, that is, $\Lambda(A) = \Lambda(B)$. Let $C$ be signless Laplacian matrix of $G'$. Now the signless Laplacian matrices of $G_1$ and $G_1^\tau$ are given by
	\begin{equation}
	Q(G_1) = \begin{bmatrix} A & 0 \\ 0 & C\end{bmatrix} ~\text{and}~ Q(G_1^\tau) = \begin{bmatrix} B & 0 \\ 0 & C\end{bmatrix}.
	\end{equation}
	From spectral properties of block matrices we come to the conclusion that $\Lambda(Q(G_1)) = \Lambda(Q(G_1^\tau))$. 
	
	According to the above procedure, the new graph $G_1$ is a disconnected graph with at least two components. One is isomorphic to its partial transpose. Another one makes $G_1$ and $G_1^\tau$ non-isomorphic. Below we generate connected graphs which are non-isomorphic, and $Q$-cospectral to their partial transpose.
	
	\begin{procedure}\label{procedure2}
		Let $G$ be a graph derived by theorem \ref{theorem1} or its corollaries which contains the edge $(v_{1, i}, v_{2, j})$ for $i \neq j$. Now add any number of pairs of edges $\{(v_{1, k}, v_{1, l}), (v_{2, k}, v_{2, l}): k, l \notin \{i, j\}\}$ with $G$. The new graph $G_1$ is $Q$-cospectral to its partial transpose. 
	\end{procedure}
	
	\begin{proof}
		Checking $Q$-cospectrality of $G_1$ and $G_1^\tau$ is similar to that of the theorem 2. Non-isomorphims is generated by the alteration of an edge $(v_{1, i}, v_{2, j}), i \neq j$ during partial transpose and non existence of $(v_{2, i}, v_{2, j})$.
	\end{proof}

	For any vertex in a $q$-circle there are $(q - 2)$ non-adjacent vertices. Hence, there are $q(q - 2)$ possible edges which may construct inside $\langle C_1 \rangle$. But one pair $v_{1, i}, v_{1, j}$ will not be considered. For any choice $(v_{1, k}, v_{1, l})$ of these $(q(q - 2) - 1)$ edges in $\langle C_1 \rangle$ we need to add $(v_{2, k}, v_{2, l})$ in $\langle C_2 \rangle$. Therefore, given any graph generated by theorem \ref{theorem1} there are at most $2^{(q(q - 2) - 1)}$ graphs constructed by procedure \ref{procedure1}, which are non-isomorphic and $Q$-cospectral to their partial transpose.
	
	\begin{example}
		The graph $G$ in the figure below in generated by theorem 2 which is non-isomorphic and $Q$-cospectral to its partial transpose. Here $i = 3$ and $j = 4$.\\
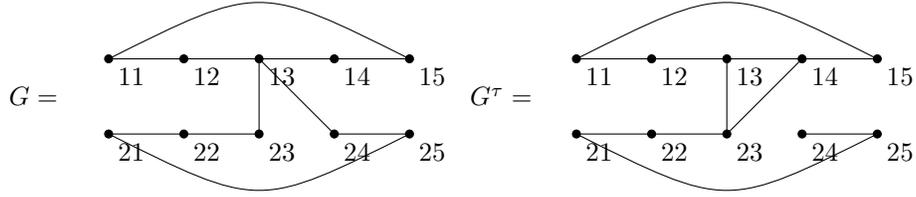
		\begin{tikzpicture}[scale = 1]
			\node at (-1, .5) {$G = $};
			\draw[fill] (0, 1) circle [radius = 1.5 pt];
			\node[below right] at (0, 1) {$11$};
			\draw[fill] (1, 1) circle [radius = 1.5 pt];
			\node[below right] at (1, 1) {$12$};
			\draw[fill] (2, 1) circle [radius = 1.5 pt];
			\node[below right] at (2, 1) {$13$};
			\draw[fill] (3, 1) circle [radius = 1.5 pt];
			\node[below right] at (3, 1) {$14$};
			\draw[fill] (4, 1) circle [radius = 1.5 pt];
			\node[below right] at (4, 1) {$15$};
			\draw[fill] (0, 0) circle [radius = 1.5 pt];
			\node[below right] at (0, 0) {$21$};
			\draw[fill] (1, 0) circle [radius = 1.5 pt];
			\node[below right] at (1, 0) {$22$};
			\draw[fill] (2, 0) circle [radius = 1.5 pt];
			\node[below right] at (2, 0) {$23$};
			\draw[fill] (3, 0) circle [radius = 1.5 pt];
			\node[below right] at (3, 0) {$24$};
			\draw[fill] (4, 0) circle [radius = 1.5 pt];
			\node[below right] at (4, 0) {$25$};
			\draw (0, 0) -- (1, 0) -- (2, 0) -- (2,1) -- (3, 0) -- (4, 0);
			\draw (0, 1) -- (1, 1) -- (2, 1) -- (3, 1) -- (4, 1);
			\draw (0, 0) .. controls (2, -1) .. (4, 0);
			\draw (0, 1) .. controls (2, 2) .. (4, 1);		
		\end{tikzpicture}
		\begin{tikzpicture}[scale = 1]
			\node at (-1, .5) {$G^\tau = $};
			\draw[fill] (0, 1) circle [radius = 1.5 pt];
			\node[below right] at (0, 1) {$11$};
			\draw[fill] (1, 1) circle [radius = 1.5 pt];
			\node[below right] at (1, 1) {$12$};
			\draw[fill] (2, 1) circle [radius = 1.5 pt];
			\node[below right] at (2, 1) {$13$};
			\draw[fill] (3, 1) circle [radius = 1.5 pt];
			\node[below right] at (3, 1) {$14$};
			\draw[fill] (4, 1) circle [radius = 1.5 pt];
			\node[below right] at (4, 1) {$15$};
			\draw[fill] (0, 0) circle [radius = 1.5 pt];
			\node[below right] at (0, 0) {$21$};
			\draw[fill] (1, 0) circle [radius = 1.5 pt];
			\node[below right] at (1, 0) {$22$};
			\draw[fill] (2, 0) circle [radius = 1.5 pt];
			\node[below right] at (2, 0) {$23$};
			\draw[fill] (3, 0) circle [radius = 1.5 pt];
			\node[below right] at (3, 0) {$24$};
			\draw[fill] (4, 0) circle [radius = 1.5 pt];
			\node[below right] at (4, 0) {$25$};
			\draw (0, 0) -- (1, 0) -- (2, 0) -- (2,1);
			\draw (3, 0) -- (4, 0);
			\draw (2, 0) -- (3, 1);
			\draw (0, 1) -- (1, 1) -- (2, 1) -- (3, 1) -- (4, 1);
			\draw (0, 0) .. controls (2, -1) .. (4, 0);
			\draw (0, 1) .. controls (2, 2) .. (4, 1);		
		\end{tikzpicture}
		We add a pair of edges $(v_{1 2}, v_{15})$ and $(v_{22}, v_{25})$ with $G$ to form $G_1$ below. It can be verified that $G_1$ and $G_1^\tau$ are non-isomorphic and $Q$-cospectral.\\
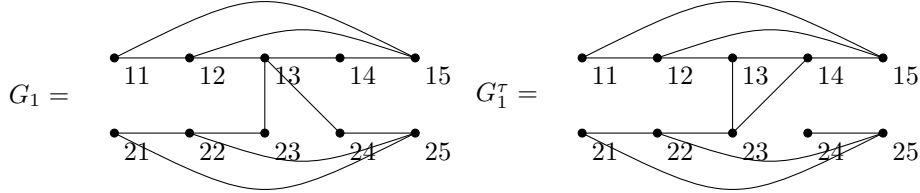
		\begin{tikzpicture}[scale = 1]
			\node at (-1, .5) {$G_1 = $};
			\draw[fill] (0, 1) circle [radius = 1.5 pt];
			\node[below right] at (0, 1) {$11$};
			\draw[fill] (1, 1) circle [radius = 1.5 pt];
			\node[below right] at (1, 1) {$12$};
			\draw[fill] (2, 1) circle [radius = 1.5 pt];
			\node[below right] at (2, 1) {$13$};
			\draw[fill] (3, 1) circle [radius = 1.5 pt];
			\node[below right] at (3, 1) {$14$};
			\draw[fill] (4, 1) circle [radius = 1.5 pt];
			\node[below right] at (4, 1) {$15$};
			\draw[fill] (0, 0) circle [radius = 1.5 pt];
			\node[below right] at (0, 0) {$21$};
			\draw[fill] (1, 0) circle [radius = 1.5 pt];
			\node[below right] at (1, 0) {$22$};
			\draw[fill] (2, 0) circle [radius = 1.5 pt];
			\node[below right] at (2, 0) {$23$};
			\draw[fill] (3, 0) circle [radius = 1.5 pt];
			\node[below right] at (3, 0) {$24$};
			\draw[fill] (4, 0) circle [radius = 1.5 pt];
			\node[below right] at (4, 0) {$25$};
			\draw (0, 0) -- (1, 0) -- (2, 0) -- (2,1) -- (3, 0) -- (4, 0);
			\draw (0, 1) -- (1, 1) -- (2, 1) -- (3, 1) -- (4, 1);
			\draw (0, 0) .. controls (2, -1) .. (4, 0);
			\draw (0, 1) .. controls (2, 2) .. (4, 1);
			\draw (1, 1) ..controls(2.5 , 1.5) .. (4, 1);
			\draw (1, 0) ..controls(2.5 , -.5) .. (4, 0);
		\end{tikzpicture}
		\begin{tikzpicture}[scale = 1]
			\node at (-1, .5) {$G_1^\tau = $};
			\draw[fill] (0, 1) circle [radius = 1.5 pt];
			\node[below right] at (0, 1) {$11$};
			\draw[fill] (1, 1) circle [radius = 1.5 pt];
			\node[below right] at (1, 1) {$12$};
			\draw[fill] (2, 1) circle [radius = 1.5 pt];
			\node[below right] at (2, 1) {$13$};
			\draw[fill] (3, 1) circle [radius = 1.5 pt];
			\node[below right] at (3, 1) {$14$};
			\draw[fill] (4, 1) circle [radius = 1.5 pt];
			\node[below right] at (4, 1) {$15$};
			\draw[fill] (0, 0) circle [radius = 1.5 pt];
			\node[below right] at (0, 0) {$21$};
			\draw[fill] (1, 0) circle [radius = 1.5 pt];
			\node[below right] at (1, 0) {$22$};
			\draw[fill] (2, 0) circle [radius = 1.5 pt];
			\node[below right] at (2, 0) {$23$};
			\draw[fill] (3, 0) circle [radius = 1.5 pt];
			\node[below right] at (3, 0) {$24$};
			\draw[fill] (4, 0) circle [radius = 1.5 pt];
			\node[below right] at (4, 0) {$25$};
			\draw (0, 0) -- (1, 0) -- (2, 0) -- (2,1);
			\draw (3, 0) -- (4, 0);
			\draw (2, 0) -- (3, 1);
			\draw (0, 1) -- (1, 1) -- (2, 1) -- (3, 1) -- (4, 1);
			\draw (0, 0) .. controls (2, -1) .. (4, 0);
			\draw (0, 1) .. controls (2, 2) .. (4, 1);
			\draw (1, 1) ..controls(2.5 , 1.5) .. (4, 1);
			\draw (1, 0) ..controls(2.5 , -.5) .. (4, 0);	
		\end{tikzpicture}
		
		We can easily verify that the induced subgraphs generated by the vertex set $\{v_{1, 3}, v_{1, 4}, v_{2, 3}, v_{2, 4}\}$ in $G, G^\tau, G_1, G_1^\tau$ are non-isomorphic. This characteristic plays a key role to make $G$ and $G_1$ non-isomorphic to their partial transposes.
	\end{example}

	Partial transpose also does not alter the partially symmetric structures inside $\langle C_\mu, C_\nu \rangle$. Therefore, we can induce partially symmetric subgraphs with $\langle C_\mu, C_\nu \rangle$ for generating new $Q$-cospectral graphs.
	
	\begin{procedure}\label{procedure3}
		Let $G$ be a graph generated by theorems \ref{theorem1}, or its corollaries, or procedure \ref{procedure2} containing the edge $(v_{1, i}, v_{2, j})$ for $i \neq j$. Add edges from the set $\{(v_{1, k}, v_{2, l}): \forall k, l \notin \{i, j\}\}$ such that the new edges construct a partial symmetric subgraph among themselves with respect to the existing vertex labellings. Then the new graph $G_1$ is $Q$-cospectral to its partial transpose.
	\end{procedure}

	\begin{proof}
		We can choose $Q$-cospectrality and non-negativity as earlier. A partial symmetric subgraph is unaltered during partial transpose. Also the newly added partially symmetric subgraph does not influence the edge $(v_{1, i}, v_{2, j})$ to generate non-isomorphic graphs $G_1$ and $G_1^\tau$.
	\end{proof}

	In this procedure, we construct a partially symmetric subgraph inside the graph $G$ to construct new graph $G_1$. In the formation of partially symmetric subgraph $(q - 2) \geq 0$ vertices of a cluster may participate. The lemma \ref{partially_symmetric_count} suggests that $2^{\frac{q - 2}{2}(3q - 7)}$ graphs may be considered by this procedure.
	
	\begin{example}
		We begin this example with a graph $G$ which is produced by procedure \ref{procedure2}. Clearly, $G$ is non-isomorphic and $Q$-cospectral to $G^\tau$, which are depicted below:\\
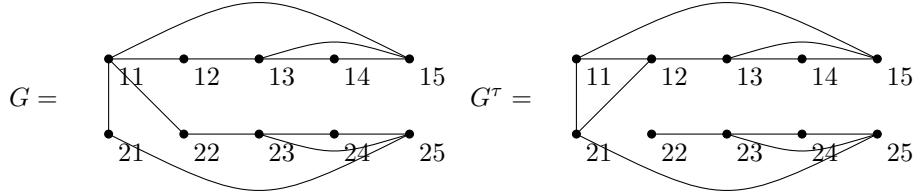
		\begin{tikzpicture}
			\node at (-1, .5) {$G = $};
			\draw[fill] (0, 1) circle [radius = 1.5 pt];
			\node[below right] at (0, 1) {$11$};
			\draw[fill] (1, 1) circle [radius = 1.5 pt];
			\node[below right] at (1, 1) {$12$};
			\draw[fill] (2, 1) circle [radius = 1.5 pt];
			\node[below right] at (2, 1) {$13$};
			\draw[fill] (3, 1) circle [radius = 1.5 pt];
			\node[below right] at (3, 1) {$14$};
			\draw[fill] (4, 1) circle [radius = 1.5 pt];
			\node[below right] at (4, 1) {$15$};
			\draw[fill] (0, 0) circle [radius = 1.5 pt];
			\node[below right] at (0, 0) {$21$};
			\draw[fill] (1, 0) circle [radius = 1.5 pt];
			\node[below right] at (1, 0) {$22$};
			\draw[fill] (2, 0) circle [radius = 1.5 pt];
			\node[below right] at (2, 0) {$23$};
			\draw[fill] (3, 0) circle [radius = 1.5 pt];
			\node[below right] at (3, 0) {$24$};
			\draw[fill] (4, 0) circle [radius = 1.5 pt];
			\node[below right] at (4, 0) {$25$};
			\draw (0, 1) -- (1, 1) -- (2, 1) -- (3, 1) -- (4, 1);
			\draw (0, 1) -- (0, 0);
			\draw (0, 1) -- (1, 0);
			\draw (1, 0) -- (2, 0) -- (3, 0) -- (4, 0);
			\draw (0 , 0) .. controls (2, -1) .. (4, 0);
			\draw (0 , 1) .. controls (2, 2) .. (4, 1);
			\draw (2 , 1) .. controls (3, 1.3) .. (4, 1);
			\draw (2 , 0) .. controls (3, -.3) .. (4, 0);
		\end{tikzpicture}
		\begin{tikzpicture}
			\node at (-1, .5) {$G^\tau = $};
			\draw[fill] (0, 1) circle [radius = 1.5 pt];
			\node[below right] at (0, 1) {$11$};
			\draw[fill] (1, 1) circle [radius = 1.5 pt];
			\node[below right] at (1, 1) {$12$};
			\draw[fill] (2, 1) circle [radius = 1.5 pt];
			\node[below right] at (2, 1) {$13$};
			\draw[fill] (3, 1) circle [radius = 1.5 pt];
			\node[below right] at (3, 1) {$14$};
			\draw[fill] (4, 1) circle [radius = 1.5 pt];
			\node[below right] at (4, 1) {$15$};
			\draw[fill] (0, 0) circle [radius = 1.5 pt];
			\node[below right] at (0, 0) {$21$};
			\draw[fill] (1, 0) circle [radius = 1.5 pt];
			\node[below right] at (1, 0) {$22$};
			\draw[fill] (2, 0) circle [radius = 1.5 pt];
			\node[below right] at (2, 0) {$23$};
			\draw[fill] (3, 0) circle [radius = 1.5 pt];
			\node[below right] at (3, 0) {$24$};
			\draw[fill] (4, 0) circle [radius = 1.5 pt];
			\node[below right] at (4, 0) {$25$};
			\draw (0, 1) -- (1, 1) -- (2, 1) -- (3, 1) -- (4, 1);
			\draw (0, 1) -- (0, 0);
			\draw (1, 1) -- (0, 0);
			\draw (1, 0) -- (2, 0) -- (3, 0) -- (4, 0);
			\draw (0 , 0) .. controls (2, -1) .. (4, 0);
			\draw (0 , 1) .. controls (2, 2) .. (4, 1);
			\draw (2 , 1) .. controls (3, 1.3) .. (4, 1);
			\draw (2 , 0) .. controls (3, -.3) .. (4, 0);
		\end{tikzpicture}
		Now we add a partially symmetric subgraph with $G$. It consists of the edge set $\{(v_{13}, v_{23}), (v_{15}, v_{25}), (v_{13}, v_{25}), (v_{15}, v_{23})\}$. The new graph $G_1$ is also non-isomorphic and $Q$-cospectral to its partial transpose, which are as follows:\\
		\begin{tikzpicture}
			\node at (-1, .5) {$G_1 = $};
			\draw[fill] (0, 1) circle [radius = 1.5 pt];
			\node[below right] at (0, 1) {$11$};
			\draw[fill] (1, 1) circle [radius = 1.5 pt];
			\node[below right] at (1, 1) {$12$};
			\draw[fill] (2, 1) circle [radius = 1.5 pt];
			\node[below right] at (2, 1) {$13$};
			\draw[fill] (3, 1) circle [radius = 1.5 pt];
			\node[below right] at (3, 1) {$14$};
			\draw[fill] (4, 1) circle [radius = 1.5 pt];
			\node[below right] at (4, 1) {$15$};
			\draw[fill] (0, 0) circle [radius = 1.5 pt];
			\node[below right] at (0, 0) {$21$};
			\draw[fill] (1, 0) circle [radius = 1.5 pt];
			\node[below right] at (1, 0) {$22$};
			\draw[fill] (2, 0) circle [radius = 1.5 pt];
			\node[below right] at (2, 0) {$23$};
			\draw[fill] (3, 0) circle [radius = 1.5 pt];
			\node[below right] at (3, 0) {$24$};
			\draw[fill] (4, 0) circle [radius = 1.5 pt];
			\node[below right] at (4, 0) {$25$};
			\draw (0, 1) -- (1, 1) -- (2, 1) -- (3, 1) -- (4, 1);
			\draw (0, 1) -- (0, 0);
			\draw (0, 1) -- (1, 0);
			\draw (1, 0) -- (2, 0) -- (3, 0) -- (4, 0);
			\draw (2, 0) -- (2, 1);
			\draw (4, 0) -- (4, 1);
			\draw (2, 0) -- (4, 1);
			\draw (4, 0) -- (2, 1);
			\draw (0 , 0) .. controls (2, -1) .. (4, 0);
			\draw (0 , 1) .. controls (2, 2) .. (4, 1);
			\draw (2 , 1) .. controls (3, 1.3) .. (4, 1);
			\draw (2 , 0) .. controls (3, -.3) .. (4, 0);
		\end{tikzpicture}
		\begin{tikzpicture}
			\node at (-1, .5) {$G_1^\tau = $};
			\draw[fill] (0, 1) circle [radius = 1.5 pt];
			\node[below right] at (0, 1) {$11$};
			\draw[fill] (1, 1) circle [radius = 1.5 pt];
			\node[below right] at (1, 1) {$12$};
			\draw[fill] (2, 1) circle [radius = 1.5 pt];
			\node[below right] at (2, 1) {$13$};
			\draw[fill] (3, 1) circle [radius = 1.5 pt];
			\node[below right] at (3, 1) {$14$};
			\draw[fill] (4, 1) circle [radius = 1.5 pt];
			\node[below right] at (4, 1) {$15$};
			\draw[fill] (0, 0) circle [radius = 1.5 pt];
			\node[below right] at (0, 0) {$21$};
			\draw[fill] (1, 0) circle [radius = 1.5 pt];
			\node[below right] at (1, 0) {$22$};
			\draw[fill] (2, 0) circle [radius = 1.5 pt];
			\node[below right] at (2, 0) {$23$};
			\draw[fill] (3, 0) circle [radius = 1.5 pt];
			\node[below right] at (3, 0) {$24$};
			\draw[fill] (4, 0) circle [radius = 1.5 pt];
			\node[below right] at (4, 0) {$25$};
			\draw (0, 1) -- (1, 1) -- (2, 1) -- (3, 1) -- (4, 1);
			\draw (0, 1) -- (0, 0);
			\draw (1, 1) -- (0, 0);
			\draw (1, 0) -- (2, 0) -- (3, 0) -- (4, 0);
			\draw (2, 0) -- (2, 1);
			\draw (4, 0) -- (4, 1);
			\draw (2, 0) -- (4, 1);
			\draw (4, 0) -- (2, 1);
			\draw (0 , 0) .. controls (2, -1) .. (4, 0);
			\draw (0 , 1) .. controls (2, 2) .. (4, 1);
			\draw (2 , 1) .. controls (3, 1.3) .. (4, 1);
			\draw (2 , 0) .. controls (3, -.3) .. (4, 0);
		\end{tikzpicture}
	\end{example}
	
	Procedure \ref{procedure2} and \ref{procedure2} increase the edges in a graph $G$ such that the new graph $G_1$ is $Q$-cospectral to its partial transpose. We can construct large families of graphs by adding both vertices and edges, which is discuss in the next procedure.
	
	\begin{procedure}\label{procedure4}
		Let $G$ be a graph generated by using any of the above theorems and procedures which has an edge $(v_{1, i}, v_{2, j})$ such that $(v_{1, j}, v_{2, i}) \notin E(G)$. Add equal number of vertices with every clusters. New edges may be constructed by performing any one or more of the following operations:
		\begin{enumerate}
			\item
				Add arbitrary set of edges joining the new vertices within the clusters.
			\item
				Edges can be added between the old and new vertices inside the cluaters, such that, the vertices $v_{1, i}, v_{1, j}, v_{2, i}$, and $v_{2, j}$ are not adjacent to any of the new vertices.
			\item
				New edges can be included between the new vertices belonging to both clusters such that they form a partially symmetric subgraph.
		\end{enumerate}
		The new graph $G_1$ is non-isomorphic and $Q$-cospectral to its partial transpose.
	\end{procedure}
	
	\begin{proof}
		One can check $G_1$ and $G_1^\tau$ are $Q$-cospectral and non-isomorphic as earlier. Note that, the induced subgraph generated by new vertices and edges is a partially symmetric subgraph which does not influence in generating non-isomorphic $Q$-cospectral pairs.
	\end{proof}

	\begin{example}
		Consider the graph $G$ depicted in the example 4. It has an edge $(v_{11}, v_{22})$ such that $(v_{12}, v_{21})$ is missing. It has six vertices arranged into two clusters. We add three new vertices to every cluster. In the cluster $\langle C_1 \rangle$ we add a tree and in cluster $\langle C_2 \rangle$ we include a 3-cycle with a hair. They are connected to vertices $v_{13}$ and $v_{23}$ which are not in $\{v_{11}, v_{22}, v_{12}, v_{21}\}$. Also, we have added an edge $(v_{14}, v_{24})$, which forms a partially symmetric subgraph in $\langle C_1, C_2 \rangle$. The resultant graph: \\
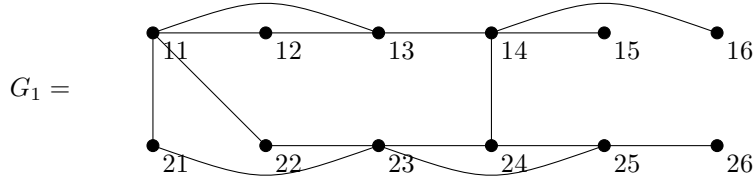
		\begin{tikzpicture}[scale = 1.5]
			\node at (-1, .5) {$G_1 = $};
			\draw[fill] (0, 1) circle [radius = 1.5 pt];
			\node[below right] at (0,1) {$11$};
			\draw[fill] (1, 1) circle [radius = 1.5 pt];
			\node[below right] at (1,1) {$12$};
			\draw[fill] (2, 1) circle [radius = 1.5 pt];
			\node[below right] at (2,1) {$13$};
			\draw[fill] (3, 1) circle [radius = 1.5 pt];
			\node[below right] at (3,1) {$14$};
			\draw[fill] (4, 1) circle [radius = 1.5 pt];
			\node[below right] at (4,1) {$15$};
			\draw[fill] (5, 1) circle [radius = 1.5 pt];
			\node[below right] at (5,1) {$16$};
			\draw[fill] (0, 0) circle [radius = 1.5 pt];
			\node[below right] at (0, 0) {$21$};
			\draw[fill] (1, 0) circle [radius = 1.5 pt];
			\node[below right] at (1,0) {$22$};
			\draw[fill] (2, 0) circle [radius = 1.5 pt];
			\node[below right] at (2,0) {$23$};
			\draw[fill] (3, 0) circle [radius = 1.5 pt];
			\node[below right] at (3, 0) {$24$};
			\draw[fill] (4, 0) circle [radius = 1.5 pt];
			\node[below right] at (4, 0) {$25$};
			\draw[fill] (5, 0) circle [radius = 1.5 pt];
			\node[below right] at (5,0) {$26$};
			\draw (0, 1) -- (1, 1) -- (2, 1);
			\draw (0, 1) -- (0, 0);
			\draw (0, 1) -- (1, 0);
			\draw (1, 0) -- (2, 0);
			\draw (0, 1) .. controls (1, 1.35) .. (2, 1);
			\draw (0, 0) .. controls (1, -.35) .. (2, 0);
			\draw (2, 1) -- (3, 1) -- (4, 1);
			\draw (2, 0) -- (3, 0) -- (4, 0);
			\draw (3, 1) .. controls (4, 1.35) .. (5, 1);
			\draw (2, 0) .. controls (3, -.35) .. (4, 0);
			\draw (4, 0) -- (5, 0);
			\draw (3, 1) -- (3, 0);
		\end{tikzpicture}\\
		It is $Q$-cospectral to its partial transpose:\\
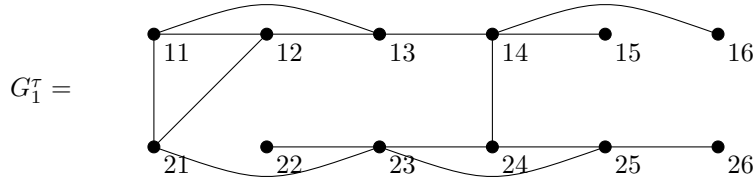
		\begin{tikzpicture}[scale = 1.5]
			\node at (-1, .5) {$G_1^\tau = $};
			\draw[fill] (0, 1) circle [radius = 1.5 pt];
			\node[below right] at (0,1) {$11$};
			\draw[fill] (1, 1) circle [radius = 1.5 pt];
			\node[below right] at (1,1) {$12$};
			\draw[fill] (2, 1) circle [radius = 1.5 pt];
			\node[below right] at (2,1) {$13$};
			\draw[fill] (3, 1) circle [radius = 1.5 pt];
			\node[below right] at (3,1) {$14$};
			\draw[fill] (4, 1) circle [radius = 1.5 pt];
			\node[below right] at (4,1) {$15$};
			\draw[fill] (5, 1) circle [radius = 1.5 pt];
			\node[below right] at (5,1) {$16$};
			\draw[fill] (0, 0) circle [radius = 1.5 pt];
			\node[below right] at (0, 0) {$21$};
			\draw[fill] (1, 0) circle [radius = 1.5 pt];
			\node[below right] at (1,0) {$22$};
			\draw[fill] (2, 0) circle [radius = 1.5 pt];
			\node[below right] at (2,0) {$23$};
			\draw[fill] (3, 0) circle [radius = 1.5 pt];
			\node[below right] at (3, 0) {$24$};
			\draw[fill] (4, 0) circle [radius = 1.5 pt];
			\node[below right] at (4, 0) {$25$};
			\draw[fill] (5, 0) circle [radius = 1.5 pt];
			\node[below right] at (5,0) {$26$};
			\draw (0, 1) -- (1, 1) -- (2, 1);
			\draw (0, 1) -- (0, 0);
			\draw (0, 0) -- (1, 1);
			\draw (1, 0) -- (2, 0);
			\draw (0, 1) .. controls (1, 1.35) .. (2, 1);
			\draw (0, 0) .. controls (1, -.35) .. (2, 0);
			\draw (2, 1) -- (3, 1) -- (4, 1);
			\draw (2, 0) -- (3, 0) -- (4, 0);
			\draw (3, 1) .. controls (4, 1.35) .. (5, 1);
			\draw (2, 0) .. controls (3, -.35) .. (4, 0);
			\draw (4, 0) -- (5, 0);
			\draw (3, 1) -- (3, 0);
		\end{tikzpicture} 
		
		As an another example, consider the following graph $G$ which is generated by the theorem 4.\\
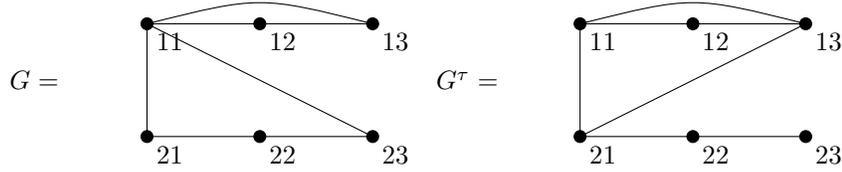
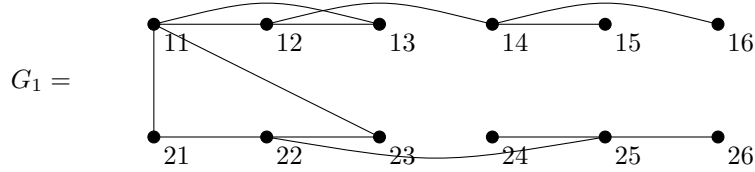
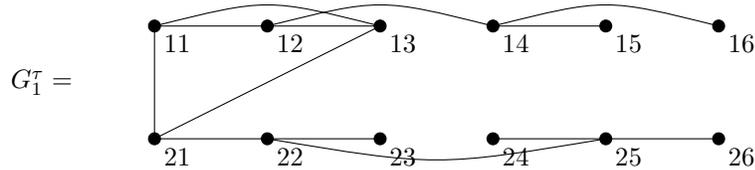
		\begin{tikzpicture}[scale = 1.5]
			\node at (0, .5) {$G = $};
			\draw[fill] (1, 1) circle [radius = 1.5pt];
			\node[below right] at (1, 1) {$11$};
			\draw[fill] (2, 1) circle [radius = 1.5pt];
			\node[below right] at (2, 1) {$12$};
			\draw[fill] (3, 1) circle [radius = 1.5pt];
			\node[below right] at (3, 1) {$13$};
			\draw[fill] (1, 0) circle [radius = 1.5pt];
			\node[below right] at (1, 0) {$21$};
			\draw[fill] (2, 0) circle [radius = 1.5pt];
			\node[below right] at (2, 0) {$22$};
			\draw[fill] (3, 0) circle [radius = 1.5pt];
			\node[below right] at (3, 0) {$23$};
			\draw (3, 0) -- (2, 0) -- (1, 0) -- (1, 1) -- (2, 1) -- (3, 1);
			\draw (1, 1) -- (3, 0);
			\draw (1,1) .. controls (2, 1.25) .. (3, 1);
		\end{tikzpicture}
		\begin{tikzpicture}[scale = 1.5]
			\node at (0, .5) {$G^\tau = $};
			\draw[fill] (1, 1) circle [radius = 1.5pt];
			\node[below right] at (1, 1) {$11$};
			\draw[fill] (2, 1) circle [radius = 1.5pt];
			\node[below right] at (2, 1) {$12$};
			\draw[fill] (3, 1) circle [radius = 1.5pt];
			\node[below right] at (3, 1) {$13$};
			\draw[fill] (1, 0) circle [radius = 1.5pt];
			\node[below right] at (1, 0) {$21$};
			\draw[fill] (2, 0) circle [radius = 1.5pt];
			\node[below right] at (2, 0) {$22$};
			\draw[fill] (3, 0) circle [radius = 1.5pt];
			\node[below right] at (3, 0) {$23$};
			\draw (3, 0) -- (2, 0) -- (1, 0) -- (1, 1) -- (2, 1) -- (3, 1);
			\draw (1, 0) -- (3, 1);
			\draw (1,1) .. controls (2, 1.25) .. (3, 1);
		\end{tikzpicture}
		
		\begin{tikzpicture}[scale = 1.5]
			\node at (0, .5) {$G_1 = $};
			\draw[fill] (1, 1) circle [radius = 1.5pt];
			\node[below right] at (1, 1) {$11$};
			\draw[fill] (2, 1) circle [radius = 1.5pt];
			\node[below right] at (2, 1) {$12$};
			\draw[fill] (3, 1) circle [radius = 1.5pt];
			\node[below right] at (3, 1) {$13$};
			\draw[fill] (4, 1) circle [radius = 1.5pt];
			\node[below right] at (4, 1) {$14$};
			\draw[fill] (5, 1) circle [radius = 1.5pt];
			\node[below right] at (5, 1) {$15$};
			\draw[fill] (6, 1) circle [radius  = 1.5pt];
			\node[below right] at (6, 1) {$16$};
			\draw[fill] (1, 0) circle [radius = 1.5pt];
			\node[below right] at (1, 0) {$21$};
			\draw[fill] (2, 0) circle [radius = 1.5pt];
			\node[below right] at (2, 0) {$22$};
			\draw[fill] (3, 0) circle [radius = 1.5pt];
			\node[below right] at (3, 0) {$23$};
			\draw[fill] (4, 0) circle [radius = 1.5pt];
			\node[below right] at (4, 0) {$24$};
			\draw[fill] (5, 0) circle [radius = 1.5pt];
			\node[below right] at (5, 0) {$25$};
			\draw[fill] (6, 0) circle [radius  = 1.5pt];
			\node[below right] at (6, 0) {$26$};
			\draw (3, 0) -- (2, 0) -- (1, 0) -- (1, 1) -- (2, 1) -- (3, 1);
			\draw (1, 1) -- (3, 0);
			\draw (1,1) .. controls (2, 1.25) .. (3, 1);
			\draw (2, 1) .. controls (3, 1.25) .. (4, 1);
			\draw (4, 1) -- (5, 1);
			\draw (4, 1) .. controls (5, 1.25) .. (6, 1);
			\draw (2, 0) .. controls (3.5, -.25) .. (5, 0);
			\draw (4, 0) -- (5, 0) -- (6, 0);
		\end{tikzpicture}
		
		\begin{tikzpicture}[scale = 1.5]
			\node at (0, .5) {$G^\tau_1 = $};
			\draw[fill] (1, 1) circle [radius = 1.5pt];
			\node[below right] at (1, 1) {$11$};
			\draw[fill] (2, 1) circle [radius = 1.5pt];
			\node[below right] at (2, 1) {$12$};
			\draw[fill] (3, 1) circle [radius = 1.5pt];
			\node[below right] at (3, 1) {$13$};
			\draw[fill] (4, 1) circle [radius = 1.5pt];
			\node[below right] at (4, 1) {$14$};
			\draw[fill] (5, 1) circle [radius = 1.5pt];
			\node[below right] at (5, 1) {$15$};
			\draw[fill] (6, 1) circle [radius  = 1.5pt];
			\node[below right] at (6, 1) {$16$};
			\draw[fill] (1, 0) circle [radius = 1.5pt];
			\node[below right] at (1, 0) {$21$};
			\draw[fill] (2, 0) circle [radius = 1.5pt];
			\node[below right] at (2, 0) {$22$};
			\draw[fill] (3, 0) circle [radius = 1.5pt];
			\node[below right] at (3, 0) {$23$};
			\draw[fill] (4, 0) circle [radius = 1.5pt];
			\node[below right] at (4, 0) {$24$};
			\draw[fill] (5, 0) circle [radius = 1.5pt];
			\node[below right] at (5, 0) {$25$};
			\draw[fill] (6, 0) circle [radius  = 1.5pt];
			\node[below right] at (6, 0) {$26$};
			\draw (3, 0) -- (2, 0) -- (1, 0) -- (1, 1) -- (2, 1) -- (3, 1);
			\draw (1, 0) -- (3, 1);
			\draw (1,1) .. controls (2, 1.25) .. (3, 1);
			\draw (2, 1) .. controls (3, 1.25) .. (4, 1);
			\draw (4, 1) -- (5, 1);
			\draw (4, 1) .. controls (5, 1.25) .. (6, 1);
			\draw (2, 0) .. controls (3.5, -.25) .. (5, 0);
			\draw (4, 0) -- (5, 0) -- (6, 0);
		\end{tikzpicture}
	\end{example}

	The graphs $K$ and $K^\tau$, depicted in the example \ref{critical_graphs}, play a key role in all these above constructions. They are subgraphs of all these graphs. But there are graphs which are non-isomorphic and $Q$-cospectral to their partial transpose but do not contain $K$ and $K^\tau$ as their subgraphs. We construct a class of these graphs in the following procedure.
	
	\begin{procedure}\label{procedure5}
		Let $G_0$ be isomorphic to its partial transpose $G_0^\tau$ by the mapping $f: V(G_0) \rightarrow V(G_0^\tau)$ defined by $f(v_{1i}) = v_{2i}$ and $f(v_{2i}) = v_{1i}$ for $i = 1, 2, \dots q$. Also, let the set of partial asymmetry $\mathcal{A}(G_0) \neq \emptyset$. Now add equal number of vertices to both the clusters of $G_0$ and perform any one or more of the following operations:
		\begin{enumerate}
			\item 
				Add arbitrary set of edges joining the new vertices within the clusters.
			\item
				Consider a vertex $v_{ik}$ which is not incident to any edge in $\mathcal{A}(G_0)$. Join $v_{ik}$ with the new vertices with arbitrary edges.
			\item
				New edges can be included between the new vertices belonging to both clusters such that they form a partially symmetric subgraph.
		\end{enumerate} 
		The new graph $G$, after performing any or more of the above changes on $G_0$, is $Q$-cospectral to its partial transpose.
	\end{procedure}
	
	\begin{proof}
		This procedure is generalization of theorem \ref{theorem1} and procedure \ref{procedure4}. We can compare TIU subgraphs of $G$ and $G^\tau$ as we have done in theorem \ref{theorem1}. Also adding new vertices and edges follows procedure \ref{procedure4}.
	\end{proof}

	\begin{example}
		The graph $G_0$ is isomorphic to its partial transpose. Note that, structures of TU subgraphs remains unaltered after and before partial transpose. Also, $\mathcal{A}(G_0) = \{(v_{12}, v_{23})\}$.\\
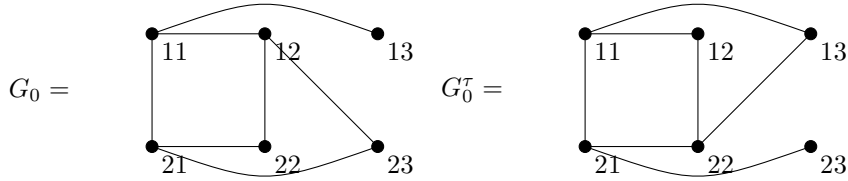
		\begin{tikzpicture}[scale = 1.5]
			\node at (0, .5) {$G_0 = $};
			\draw[fill] (1, 1) circle [radius = 1.5 pt];
			\node[below right] at (1, 1) {$11$};
			\draw[fill] (2, 1) circle [radius = 1.5 pt];
			\node[below right] at (2, 1) {$12$};
			\draw[fill] (3, 1) circle [radius = 1.5 pt];
			\node[below right] at (3, 1) {$13$};
			\draw[fill] (1, 0) circle [radius = 1.5 pt];
			\node[below right] at (1, 0) {$21$};
			\draw[fill] (2, 0) circle [radius = 1.5 pt];
			\node[below right] at (2, 0) {$22$};
			\draw[fill] (3, 0) circle [radius = 1.5 pt];
			\node[below right] at (3, 0) {$23$};
			\draw (1, 1) -- (2, 1) -- (2, 0) -- (1, 0) -- (1, 1);
			\draw (2, 1) -- (3, 0);
			\draw (1, 1) .. controls (2, 1.35) .. (3, 1);
			\draw (1, 0) .. controls (2, -.35) .. (3, 0);
		\end{tikzpicture}
		\begin{tikzpicture}[scale = 1.5]
			\node at (0, .5) {$G^\tau_0 = $};
			\draw[fill] (1, 1) circle [radius = 1.5 pt];
			\node[below right] at (1, 1) {$11$};
			\draw[fill] (2, 1) circle [radius = 1.5 pt];
			\node[below right] at (2, 1) {$12$};
			\draw[fill] (3, 1) circle [radius = 1.5 pt];
			\node[below right] at (3, 1) {$13$};
			\draw[fill] (1, 0) circle [radius = 1.5 pt];
			\node[below right] at (1, 0) {$21$};
			\draw[fill] (2, 0) circle [radius = 1.5 pt];
			\node[below right] at (2, 0) {$22$};
			\draw[fill] (3, 0) circle [radius = 1.5 pt];
			\node[below right] at (3, 0) {$23$};
			\draw (1, 1) -- (2, 1) -- (2, 0) -- (1, 0) -- (1, 1);
			\draw (2, 0) -- (3, 1);
			\draw (1, 1) .. controls (2, 1.35) .. (3, 1);
			\draw (1, 0) .. controls (2, -.35) .. (3, 0);
		\end{tikzpicture}\\
		For simplicity, we add a node to both the clusters. We add $v_{11}$ to the new node in the cluster $C_1$. The new graph $G$ and its partial transpose $G^\tau$ are non-isomorphic and $Q$-cospectral, which are depicted below:\\
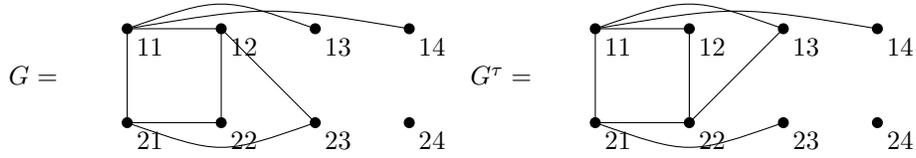
		\begin{tikzpicture}[scale = 1.25]
			\node at (0, .5) {$G = $};
			\draw[fill] (1, 1) circle [radius = 1.5 pt];
			\node[below right] at (1, 1) {$11$};
			\draw[fill] (2, 1) circle [radius = 1.5 pt];
			\node[below right] at (2, 1) {$12$};
			\draw[fill] (3, 1) circle [radius = 1.5 pt];
			\node[below right] at (3, 1) {$13$};
			\draw[fill] (4, 1) circle [radius = 1.5 pt];
			\node[below right] at (4, 1) {$14$};
			\draw[fill] (1, 0) circle [radius = 1.5 pt];
			\node[below right] at (1, 0) {$21$};
			\draw[fill] (2, 0) circle [radius = 1.5 pt];
			\node[below right] at (2, 0) {$22$};
			\draw[fill] (3, 0) circle [radius = 1.5 pt];
			\node[below right] at (3, 0) {$23$};
			\draw[fill] (4, 0) circle [radius = 1.5 pt];
			\node[below right] at (4, 0) {$24$};
			\draw (1, 1) .. controls (2.5, 1. 25) .. (4, 1);
			\draw (1, 1) -- (2, 1) -- (2, 0) -- (1, 0) -- (1, 1);
			\draw (2, 1) -- (3, 0);
			\draw (1, 1) .. controls (2, 1.35) .. (3, 1);
			\draw (1, 0) .. controls (2, -.35) .. (3, 0);
		\end{tikzpicture}
		\begin{tikzpicture}[scale = 1.25]
			\node at (0, .5) {$G^\tau = $};
			\draw[fill] (1, 1) circle [radius = 1.5 pt];
			\node[below right] at (1, 1) {$11$};
			\draw[fill] (2, 1) circle [radius = 1.5 pt];
			\node[below right] at (2, 1) {$12$};
			\draw[fill] (3, 1) circle [radius = 1.5 pt];
			\node[below right] at (3, 1) {$13$};
			\draw[fill] (4, 1) circle [radius = 1.5 pt];
			\node[below right] at (4, 1) {$14$};
			\draw[fill] (1, 0) circle [radius = 1.5 pt];
			\node[below right] at (1, 0) {$21$};
			\draw[fill] (2, 0) circle [radius = 1.5 pt];
			\node[below right] at (2, 0) {$22$};
			\draw[fill] (3, 0) circle [radius = 1.5 pt];
			\node[below right] at (3, 0) {$23$};
			\draw[fill] (4, 0) circle [radius = 1.5 pt];
			\node[below right] at (4, 0) {$24$};
			\draw (1, 1) .. controls (2.5, 1. 25) .. (4, 1);
			\draw (1, 1) -- (2, 1) -- (2, 0) -- (1, 0) -- (1, 1);
			\draw (2, 0) -- (3, 1);
			\draw (1, 1) .. controls (2, 1.35) .. (3, 1);
			\draw (1, 0) .. controls (2, -.35) .. (3, 0);
		\end{tikzpicture}\\
	\end{example}

	\section{Problems in future}
		
		The above discussion shows that partial transpose provides an useful tool in generating pair of non-isomorphic $Q$-cospectral graphs. One main challenge in this direction is to find out the vertex labelling such that $G$ and $G^\tau$ remains cospectral. Interested reader may try to construct non-isomorphic pair of normalised Laplacian cospectral graphs using this method.
		
		There are many other graphs which are $Q$-cospectral to their partial transpose, but do not follow the patterns, which we have discussed in the last two sections. Below we provide some of their examples. Interested readers may construct many such pairs of $Q$-cospectral graphs. Some of them we discuss below:
		\begin{enumerate}
			\item 
			The following graph $G$ is non-isomorphic and $Q$-cospectral to its partial transpose. Removing any or both of the edges $(v_{13}, v_{14})$ and $(v_{23}, v_{24})$ the resultant graph is non-isomorphic and $Q$-cospectral to its partial transpose.\\
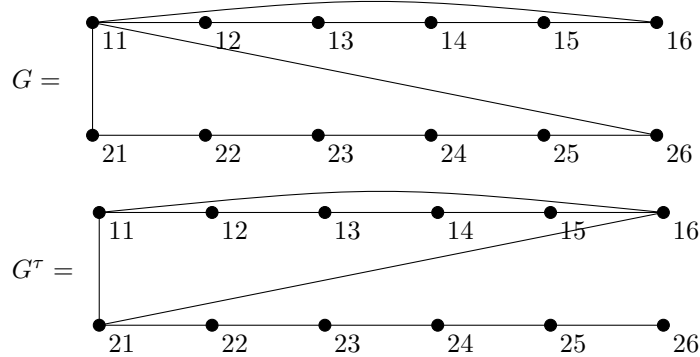
			\begin{tikzpicture}[scale = 1.5]
			\node at (.5, .5) {$G = $};
			\draw[fill] (1, 1) circle [radius = 1.5pt];
			\node[below right] at (1, 1) {$11$};
			\draw[fill] (2, 1) circle [radius = 1.5pt];
			\node[below right] at (2, 1) {$12$};
			\draw[fill] (3, 1) circle [radius = 1.5pt];
			\node[below right] at (3, 1) {$13$};
			\draw[fill] (4, 1) circle [radius = 1.5pt];
			\node[below right] at (4, 1) {$14$};
			\draw[fill] (5, 1) circle [radius = 1.5pt];
			\node[below right] at (5, 1) {$15$};
			\draw[fill] (6, 1) circle [radius  = 1.5pt];
			\node[below right] at (6, 1) {$16$};
			\draw[fill] (1, 0) circle [radius = 1.5pt];
			\node[below right] at (1, 0) {$21$};
			\draw[fill] (2, 0) circle [radius = 1.5pt];
			\node[below right] at (2, 0) {$22$};
			\draw[fill] (3, 0) circle [radius = 1.5pt];
			\node[below right] at (3, 0) {$23$};
			\draw[fill] (4, 0) circle [radius = 1.5pt];
			\node[below right] at (4, 0) {$24$};
			\draw[fill] (5, 0) circle [radius = 1.5pt];
			\node[below right] at (5, 0) {$25$};
			\draw[fill] (6, 0) circle [radius  = 1.5pt];
			\node[below right] at (6, 0) {$26$};
			\draw (1, 1) .. controls (3.5, 1.25) .. (6, 1);
			\draw (6, 1) -- (1, 1) -- (1, 0) --(6, 0);
			\draw (1, 1) -- (6, 0);
			\end{tikzpicture}
			
			\begin{tikzpicture}[scale = 1.5]
			\node at (.5, .5) {$G^\tau = $};
			\draw[fill] (1, 1) circle [radius = 1.5pt];
			\node[below right] at (1, 1) {$11$};
			\draw[fill] (2, 1) circle [radius = 1.5pt];
			\node[below right] at (2, 1) {$12$};
			\draw[fill] (3, 1) circle [radius = 1.5pt];
			\node[below right] at (3, 1) {$13$};
			\draw[fill] (4, 1) circle [radius = 1.5pt];
			\node[below right] at (4, 1) {$14$};
			\draw[fill] (5, 1) circle [radius = 1.5pt];
			\node[below right] at (5, 1) {$15$};
			\draw[fill] (6, 1) circle [radius  = 1.5pt];
			\node[below right] at (6, 1) {$16$};
			\draw[fill] (1, 0) circle [radius = 1.5pt];
			\node[below right] at (1, 0) {$21$};
			\draw[fill] (2, 0) circle [radius = 1.5pt];
			\node[below right] at (2, 0) {$22$};
			\draw[fill] (3, 0) circle [radius = 1.5pt];
			\node[below right] at (3, 0) {$23$};
			\draw[fill] (4, 0) circle [radius = 1.5pt];
			\node[below right] at (4, 0) {$24$};
			\draw[fill] (5, 0) circle [radius = 1.5pt];
			\node[below right] at (5, 0) {$25$};
			\draw[fill] (6, 0) circle [radius  = 1.5pt];
			\node[below right] at (6, 0) {$26$};
			\draw (1, 1) .. controls (3.5, 1.25) .. (6, 1);
			\draw (6, 1) -- (1, 1) -- (1, 0) -- (6, 0);
			\draw (1, 0) -- (6, 1);
			\end{tikzpicture}
			\item
			Similarly, the graph depicted below is non-isomorphic, $Q$-cospectral to its partial transpose. After removing all the edges $(v_{12}, v_{13}), (v_{13}, v_{14}), (v_{22}, v_{23})$ and $(v_{23}, v_{24})$ the new graphs are non-isomorphic, $Q$-cospectral to their partial transpose. Note that, removing less than four of those edges do not generate such pairs.\\
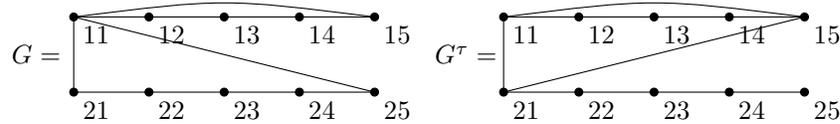
			\begin{tikzpicture}[scale = 1]
			\node at (.5, .5) {$G = $};
			\draw[fill] (1, 1) circle [radius = 1.5pt];
			\node[below right] at (1, 1) {$11$};
			\draw[fill] (2, 1) circle [radius = 1.5pt];
			\node[below right] at (2, 1) {$12$};
			\draw[fill] (3, 1) circle [radius = 1.5pt];
			\node[below right] at (3, 1) {$13$};
			\draw[fill] (4, 1) circle [radius = 1.5pt];
			\node[below right] at (4, 1) {$14$};
			\draw[fill] (5, 1) circle [radius = 1.5pt];
			\node[below right] at (5, 1) {$15$};
			\draw[fill] (1, 0) circle [radius = 1.5pt];
			\node[below right] at (1, 0) {$21$};
			\draw[fill] (2, 0) circle [radius = 1.5pt];
			\node[below right] at (2, 0) {$22$};
			\draw[fill] (3, 0) circle [radius = 1.5pt];
			\node[below right] at (3, 0) {$23$};
			\draw[fill] (4, 0) circle [radius = 1.5pt];
			\node[below right] at (4, 0) {$24$};
			\draw[fill] (5, 0) circle [radius = 1.5pt];
			\node[below right] at (5, 0) {$25$};
			\draw (1, 1) .. controls (3, 1.25) .. (5, 1);
			\draw (5, 1) -- (1, 1) -- (1, 0) --(5, 0);
			\draw (1, 1) -- (5, 0);
			\end{tikzpicture}
			\begin{tikzpicture}[scale = 1]
			\node at (.5, .5) {$G^\tau = $};
			\draw[fill] (1, 1) circle [radius = 1.5pt];
			\node[below right] at (1, 1) {$11$};
			\draw[fill] (2, 1) circle [radius = 1.5pt];
			\node[below right] at (2, 1) {$12$};
			\draw[fill] (3, 1) circle [radius = 1.5pt];
			\node[below right] at (3, 1) {$13$};
			\draw[fill] (4, 1) circle [radius = 1.5pt];
			\node[below right] at (4, 1) {$14$};
			\draw[fill] (5, 1) circle [radius = 1.5pt];
			\node[below right] at (5, 1) {$15$};
			\draw[fill] (1, 0) circle [radius = 1.5pt];
			\node[below right] at (1, 0) {$21$};
			\draw[fill] (2, 0) circle [radius = 1.5pt];
			\node[below right] at (2, 0) {$22$};
			\draw[fill] (3, 0) circle [radius = 1.5pt];
			\node[below right] at (3, 0) {$23$};
			\draw[fill] (4, 0) circle [radius = 1.5pt];
			\node[below right] at (4, 0) {$24$};
			\draw[fill] (5, 0) circle [radius = 1.5pt];
			\node[below right] at (5, 0) {$25$};
			\draw (1, 1) .. controls (3, 1.25) .. (5, 1);
			\draw (5, 1) -- (1, 1) -- (1, 0) -- (5, 0);
			\draw (1, 0) -- (5, 1);
			\end{tikzpicture}
			\item
			The following pairs of graphs are also non-isomorphic and $Q$-cospectral determined by partial transpose.\\
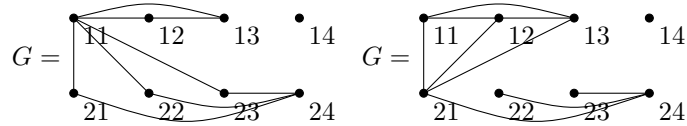
			\begin{tikzpicture}[scale = 1]
			\node at (.5, .5) {$G = $};
			\draw[fill] (1, 1) circle [radius = 1.5pt];
			\node[below right] at (1, 1) {$11$};
			\draw[fill] (2, 1) circle [radius = 1.5pt];
			\node[below right] at (2, 1) {$12$};
			\draw[fill] (3, 1) circle [radius = 1.5pt];
			\node[below right] at (3, 1) {$13$};
			\draw[fill] (4, 1) circle [radius = 1.5pt];
			\node[below right] at (4, 1) {$14$};
			\draw[fill] (1, 0) circle [radius = 1.5pt];
			\node[below right] at (1, 0) {$21$};
			\draw[fill] (2, 0) circle [radius = 1.5pt];
			\node[below right] at (2, 0) {$22$};
			\draw[fill] (3, 0) circle [radius = 1.5pt];
			\node[below right] at (3, 0) {$23$};
			\draw[fill] (4, 0) circle [radius = 1.5pt];
			\node[below right] at (4, 0) {$24$};
			\draw (1, 1) .. controls (2, 1.25) .. (3, 1);
			\draw (3, 1) -- (1, 1) -- (1, 0);
			\draw (1, 1) -- (2, 0);
			\draw (1, 1) -- (3, 0);
			\draw (3, 0) -- (4, 0);
			\draw (1, 0) .. controls (2.5, -.5) .. (4, 0);
			\draw (2, 0) .. controls (3, -.25) .. (4, 0);
			\end{tikzpicture}
			\begin{tikzpicture}[scale = 1]
			\node at (.5, .5) {$G = $};
			\draw[fill] (1, 1) circle [radius = 1.5pt];
			\node[below right] at (1, 1) {$11$};
			\draw[fill] (2, 1) circle [radius = 1.5pt];
			\node[below right] at (2, 1) {$12$};
			\draw[fill] (3, 1) circle [radius = 1.5pt];
			\node[below right] at (3, 1) {$13$};
			\draw[fill] (4, 1) circle [radius = 1.5pt];
			\node[below right] at (4, 1) {$14$};
			\draw[fill] (1, 0) circle [radius = 1.5pt];
			\node[below right] at (1, 0) {$21$};
			\draw[fill] (2, 0) circle [radius = 1.5pt];
			\node[below right] at (2, 0) {$22$};
			\draw[fill] (3, 0) circle [radius = 1.5pt];
			\node[below right] at (3, 0) {$23$};
			\draw[fill] (4, 0) circle [radius = 1.5pt];
			\node[below right] at (4, 0) {$24$};
			\draw (1, 1) .. controls (2, 1.25) .. (3, 1);
			\draw (3, 1) -- (1, 1) -- (1, 0);
			\draw (1, 0) -- (2, 1);
			\draw (1, 0) -- (3, 1);
			\draw (3, 0) -- (4, 0);
			\draw (1, 0) .. controls (2.5, -.5) .. (4, 0);
			\draw (2, 0) .. controls (3, -.25) .. (4, 0);
			\end{tikzpicture}
		\end{enumerate}

	\section*{Acknowledgement}
		
		The author is thankful to Dr. Bibhas Adhikari, and Prof. Ravindra B. Bapat for a number of discussions.
		
%	\bibliographystyle{unsrt}
%	\bibliography{library}

\end{document}